%% file: main.tex
\documentclass[11pt]{amsart}
\input{commands.tex}

\title{Twisted \texorpdfstring{$\mathbb{A}^1$}{A1}-degrees and Wronski maps}
\author{Thomas Brazelton and Jake Levinson}

\begin{document}

\begin{abstract}
The Wronski map is a central object of study in osculating Schubert calculus, and has been the focus of much work in real algebraic geometry over the last few decades due to its remarkable reality properties. In this work, we study it over arbitrary fields through the lens of quadratically-enriched enumerative geometry.  A basic obstruction is that the Wronski map is not oriented in general, so existing approaches do not apply directly. To that end, we introduce \emph{twisted $\A^1$-degrees} of finite flat lci maps of schemes, which generalize ordinary $\A^1$-degrees by allowing us to twist an orientation by the equation of a divisor, and $\A^1$-\emph{chambers}, which generalize the notion of connected component of a real manifold.

We then show that a certain weighted count of the fibers of the Wronski map depends only on the action of the absolute Galois group of the base field on the roots of the Wronski polynomial. This unifies Eremenko and Gabrielov's work on real degrees of Wronski maps, Levinson and Purbhoo's work on the character orientation of the real Wronski map, and Brazelton's work on enriched degrees of Wronski maps into one framework.
\end{abstract}
\maketitle

\input{sections/intro.tex}
\input{sections/background_a1.tex}
\input{sections/twisted_degrees.tex}
\input{sections/osculating_schubert_calc.tex}

\input{sections/twisted_wr.tex}
\input{sections/examples.tex}

\nocite{*}
\bibliography{citations.bib}
\bibliographystyle{amsalpha}
\end{document}

%% file: commands.tex
\usepackage{%
  mathtools,  
  amssymb,    
  amsthm,     
  amsfonts,   
  thmtools,   
  graphicx,
  float,
  color,
  xcolor,
  tikz,
  tikz-cd,
  mathrsfs,   
  etoolbox    
}

\usepackage{tabularx}
\newcolumntype{Y}{>{\centering\arraybackslash}X}

\definecolor{darkred}{rgb}{0.75,0,0}
\def\customcitecolor{darkred}
\def\customlinkcolor{darkred}

\usepackage[%
    colorlinks,
    citecolor=\customcitecolor,%
    linkcolor=\customlinkcolor,%
    urlcolor=\customlinkcolor%
]{hyperref}

\usepackage[capitalise,nameinlink,noabbrev]{cleveref}

\usepackage[margin=1in]{geometry}

\usepackage[final,nopatch=footnote]{microtype}

\theoremstyle{definition}
\newtheorem{theorem}{Theorem}[section]
\numberwithin{theorem}{section} 
\numberwithin{equation}{section} 

\usepackage{mfirstuc}
\newcommand{\capitalizename}[1]{\makefirstuc{#1}}

\newcommand{\defthm}[1]{%
  \newtheorem{#1}[theorem]{\capitalizename{#1}}%
}

\newcommand{\defthms}[1]{%
  \forcsvlist{\defthm}{#1}%
}

\defthms{%
  answer,assumption,claim,conjecture,construction,corollary,
  counterexample,definition,digression,discussion,example,
  examples,exercise,fact,goal,idea,intuition,lemma,
  motivation,notation,note,proposition,question,remark,setup,
  slogan,strategy,terminology,upshot%
}

\Crefname{notation}{Notation}{Notations}
\Crefname{question}{Question}{Questions}
\Crefname{terminology}{Terminology}{Terminologies}

\newcommand{\deftextcommand}[1]{%
  \expandafter\providecommand\csname #1\endcsname{\mathrm{#1}}%
}
\newcommand{\deftextcommands}[1]{%
    \forcsvlist{\deftextcommand}{#1}%
}

\deftextcommands{ab,alg,an,ann,Aut,BG,BGL,Bl,BO,BP,BSL,BSO,BSp,BSU,BU,can,cd,cdh,CH,cl,coBar,codim,codom,coeq,coev,cof,cofib,coker,colim,coim,cone,conj,const,coTor,cyc,diag,Desc,dg,Disc,disc,Div,dual,eff,EKL,End,eq,ess,et,Et,EU,ev,Ex,ex,Exc,Ext,fib,Fix,Fl,fppf,fpqc,Frac,Frob,Fun,Gal,gen,GL,gp,Gr,gr,GW,Her,Ho,hocofib,hocolim,hofib,holim,Hom,id,Idem,im,incl,Ind,ind,inj,Inn,Inv,inv,iso,Jac,KGL,kgl,KH,KO,ko,KQ,kq,KR,KSp,KU,ku,Lan,Map,map,MGL,MO,Mor,mor,MSL,MSO,MSp,MSU,MU,mult,MUP,Nm,ob,obj,op,Orb,ord,Out,perf,Perm,PGL,Pic,pr,pre,Proj,proj,prom,PSL,quot,Ran,rank,Res,res,RO,sep,sgn,SH,sig,Sing,SL,SO,soc,Sp,Span,Spec,Spin,spn,Sq,st,Stab,SU,supp,Supp,Syl,syl,Sym,syn,SYT,TC,td,Th,THH,Tor,Tot,TP,TR,Tr,tr,triv,univ,var,veff,vol,Wel,Wr,Zar}

\deftextcommands{Ab,Aff,Alg,Ani,Bimod,CAlg,Cat,CDGA,CG,CGWH,Ch,CMon,coAlg,Coh,CommRing,ConjSub,coMod,Cor,Corr,CoSh,CRing,CW,Field,Fin,FinSet,Gpd,Grp,Grpd,Grph,Kan,Kar,LMod,Mfld,Mod,NAlg,Ouv,Perf,Poset,Pr,Pre,PSh,PShv,qCat,QCoh,Rep,Ring,RMod,sAb,Sch,Set,SH,Sh,Shv,Sm,Sp,Spc,Spectra,sPre,sSet,sShv,Stack,Sub,Top,Tors,Var,Vect}

\newcommand{\defblackboardletter}[1]{%
  \expandafter\providecommand\csname #1\endcsname{\mathbb{#1}}
}
\newcommand{\defblackboardletters}[1]{%
  \forcsvlist{\defblackboardletter}{#1}%
}

\defblackboardletters{A,C,F,P,Q,R,Z}

\providecommand{\pb}{\arrow[dr,phantom,"\lrcorner" very near start]}

\providecommand{\xto}[1]{\xrightarrow{#1}}

\providecommand{\hookto}{\xhookrightarrow{}}

\providecommand{\tto}{\twoheadrightarrow}

\makeatletter
\newcommand*{\da@rightarrow}{\mathchar"0\hexnumber@\symAMSa 4B }
\newcommand*{\da@leftarrow}{\mathchar"0\hexnumber@\symAMSa 4C }
\newcommand*{\xdashto}[2][]{%
  \mathrel{%
    \mathpalette{\da@xarrow{#1}{#2}{}\da@rightarrow{\,}{}}{}%
  }%
}
\newcommand*{\da@xarrow}[7]{%
  \sbox0{$\ifx#7\scriptstyle\scriptscriptstyle\else\scriptstyle\fi#5#1#6\m@th$}%
  \sbox2{$\ifx#7\scriptstyle\scriptscriptstyle\else\scriptstyle\fi#5#2#6\m@th$}%
  \sbox4{$#7\dabar@\m@th$}%
  \dimen@=\wd0 %
  \ifdim\wd2 >\dimen@
    \dimen@=\wd2 %
  \fi
  \count@=2 %
  \def\da@bars{\dabar@\dabar@}%
  \@whiledim\count@\wd4<\dimen@\do{%
    \advance\count@\@ne
    \expandafter\def\expandafter\da@bars\expandafter{%
      \da@bars
      \dabar@ 
    }%
  }%
  \mathrel{#3}%
  \mathrel{%
    \mathop{\da@bars}\limits
    \ifx\\#1\\%
    \else
      _{\copy0}%
    \fi
    \ifx\\#2\\%
    \else
      ^{\copy2}%
    \fi
  }%
  \mathrel{#4}%
}
\makeatother

\RequirePackage{bbm}

\RequirePackage{pict2e,picture}
\makeatletter
\DeclareRobustCommand{\DDelta}{{\mathpalette\bb@Delta\relax}}
\newcommand{\bb@Delta}[2]{%
  \begingroup
  \sbox\z@{$\m@th#1\Delta$}%
  \dimendef\Dht=6 \dimendef\Dwd=8
  \setlength{\Dwd}{\wd\z@}%
  \setlength{\Dht}{\ht\z@}%
  \begin{picture}(\Dwd,\Dht)
  \put(0,0){$\m@th#1\Delta$}
  \put(.42\Dwd,.7\Dht){\line(10,-26){.25\Dht}}
  \end{picture}%
  \endgroup
}

\let\emptyset\varnothing

\usepackage[centertableaux]{ytableau}

\newcommand{\smallydiagram}[1]{{%
    \ytableausetup{boxsize=0.3em}%
    \ydiagram{#1}%
}}

\newcommand{\medydiagram}[1]{%
    \ytableausetup{boxsize=0.5em}%
    \ydiagram{#1}
}

\newcommand{\horizdomino}{\smallydiagram{2}}
\newcommand{\vertdomino}{\smallydiagram{1,1}}

\renewcommand{\square}{\smallydiagram{1}}
\newcommand{\rect}{\smallydiagram{3,3}}

\newcommand{\medhorizdomino}{\medydiagram{2}}
\newcommand{\medvertdomino}{\medydiagram{1,1}}

\newcommand{\medsquare}{\medydiagram{1}}

\newcommand{\Zdomino}{Z_{\vertdomino}}
\newcommand{\rhodomino}{\rho_{\ytableausetup{boxsize=0.2em}\ydiagram{1,1}\ytableausetup{boxsize=0.5em}}}
\newcommand{\sdomino}{s_{\ytableausetup{boxsize=0.2em}\ydiagram{1,1}\ytableausetup{boxsize=0.5em}}}

\usetikzlibrary{arrows.meta,decorations.markings}

\tikzset{
  orient/.style={
    decoration={markings,
      mark=between positions {0.5*#1} and {\pgfdecoratedpathlength} step {#1}
           with {\arrow{Stealth[length=2.5mm]}}},
    postaction={decorate}},
  orient/.default=2cm}

\renewcommand{\P}{\mathbb{P}}
\newcommand{\Osc}{\mathrm{Osc}}
\newcommand{\poly}[2]{\mathscr{P}_{#1}(#2)}
\newcommand{\sheafGW}{\mathcal{G}\mathcal{W}}

%% file: sections/intro.tex
\section{Introduction}

In the past decade, new techniques have emerged that import techniques from motivic homotopy theory into enumerative algebraic geometry. Common to these results is the idea that, while the principle of conservation of number breaks for enumerative problems over a non-closed field, one can often obtain a conserved quantity by weighting each solution by a quadratic form. This program is called \emph{quadratically enriched enumerative geometry}. Landmark results in this direction include the enriched count of the 27 lines on a cubic surface \cite{KW27} and the recent development of quadratically enriched Gromov--Witten invariants \cite{KLSW-count}.

One technique used in this program is to associate to a map $f : X \to Y$ (say a finite map of smooth varieties) a so-called \emph{$\A^1$-degree} $\deg(f)$, taking values in the Grothendieck--Witt sheaf $\sheafGW(Y)$. The $\A^1$-degree can be interpreted as counting the points of a fiber $f^{-1}(y)$, weighted by the square classes of the Jacobian determinant of $f$ along the fiber (see \cite[3.10]{KLSW-count} or \Cref{prop:sum-of-local-degrees}). To carry out this procedure, two important requirements are:
\begin{itemize}
\item $f$ must be \emph{orientable}, that is, the bundle of relative differentials $\omega_f$ must be isomorphic to the square of a line bundle,
\item for the $\A^1$-degree count to be independent of the choice of fiber, $Y$ must be $\A^1$-connected.
\end{itemize}
These are quadratically-enriched analogues of the classical fact that the Brouwer degree of a map of real manifolds is defined only when the map is oriented, and is constant only over a connected base.

This paper makes two contributions. In the first part of the paper, we show how to relax the orientability and $\A^1$-connectivity assumptions above. We introduce the notions of \emph{orientation of a map relative to a divisor} (\Cref{def:orientation-relative-to-divisor}), \emph{twisted $\A^1$-degree} (\Cref{def:twisted-degree}) and \emph{$\A^1$-chamber} (\Cref{def:A1-chamber,def:relative-A1-chamber}). These definitions allow us to use `twisted' orientations $\rho$ of $f$, essentially twisting the local weighting of each point of the fiber by the equation $s$ of a Cartier divisor, obtaining a new class $\deg^s(f, \rho) \in \sheafGW(Y)$. We show:

\begin{proposition}[Twisted local degree formula]\label{prop:twisted-degree-sum-of-local-degrees-intro}
Let $\rho$ be an orientation of $f$ relative to a Cartier divisor $D$ and let $s$ be an equation for $D$. Let $y \in Y \setminus f(D)$ be a point over which $f$ is \'etale. Then the twisted degree is given by
\[
    \deg^s(f,\rho)(y) = \sum_{x\in f^{-1}(y)} \Tr_{k(x) / k(y)} \langle \Jac(f)(x)\cdot s(x)\rangle.
\]
This is an equality in the Grothendieck--Witt group of the residue field $k(y)$.
\end{proposition}

We then prove the following fiber invariance statement within $\A^1$-chambers of $f$ relative to $D$:

\begin{theorem}[{\Cref{thm:fiber-invariance-twisted-degrees-on-chambers}}] Let $f\colon X \to Y$ be a finite flat lci morphism, let $D\subseteq X$ be a Cartier divisor with equation $s$, and let $\rho$ be a $D$-orientation of $f$. Let $g\colon W \to Y$ be an $\A^1$-chamber of $f$ relative to $D$, and let $y_1,y_2\in (Y\setminus f(D))(k)$ be in the image of $W(k)$ by $g$. Then their twisted degrees are equal:
\[
\deg^s(f,\rho)(y_1) = \deg^s(f,\rho)(y_2)\text{ in }\GW(k).
\]
\end{theorem}

Notably, even if $f$ is not orientable (in the usual sense), it may possess a natural twisted orientation and so have natural twisted $\A^1$-degrees and associated conserved quantities within $\A^1$-chambers.

In the second part of the paper, we apply these notions to define twisted $\A^1$-degrees of Wronski maps, which arise in the classical geometry of curves and in Schubert calculus. The \emph{Wronskian} of polynomials $f_1, \dots, f_d \in k[t]_{\leq d+m-1}$ is the determinant $\Wr(f_1, \ldots, f_d) = \det \big( \frac{\partial f_i}{\partial t^j} \big)$. Up to scalar, it depends only on the linear span of $f_1, \ldots, f_d$, and has degree at most $dm$, and so yields a map of varieties
\[\Wr : \Gr(d, d+m) \to \P^{dm}.\]
The Wronski map is finite and flat, but a straightforward calculation (\Cref{lem:wronski-relatively-oriented-relative-to-vertdomino}) shows that it need not be orientable; specifically, it is orientable if and only if $d$ or $m$ is odd. The Wronski map is particularly well-known in relation to the \emph{Shapiro--Shapiro conjecture}, one of several conjectures and theorems related to real Schubert calculus (see e.g.~\cite{sottileFrontiersRealitySchubert2009} for an overview). Most famous is the statement, proven by Mukhin--Tarasov--Varchenko \cite{MTV} (and previously by Eremenko--Gabrielov in the case $d=2$ \cite{eremenkoWronskiMapGrassmannians2001}), that if $f \in \mathbb{C}[t]_{\leq dm}$ is a polynomial with all real roots, then the fiber $\Wr^{-1}(f)$ consists entirely of real points. These results have been generalized in many directions, including to orthogonal Grassmannians \cite{Pur11, GLP23}, to results in algebraic combinatorics \cite{purbhooJeuTaquinMonodromy2010}, moduli spaces of curves \cite{Speyer-Schubert}, K-theory \cite{Lev17, GL17}, and positivity of solutions \cite{KMT25,karp2026universalpluckercoordinateswronski}, while other extensions remain open (see e.g. \cite{Sot00, HHMSZ15}). Related work also includes \cite{kharlamovMaximallyInflectedReal2003,HSZ17,cotterillArithmeticInflectionFormulae2023,cotterillArithmeticInflectionSuperelliptic2025}.

The most direct precursors to our paper are those focused on the \emph{degree} of the Wronski map. It is classically known that its complex (or algebraic) degree is the number of standard Young tableaux of rectangular shape $d \times m$, i.e. the degree of the Grassmannian in its Pl\"ucker embedding. In \cite{eremenkoDegreesRealWronski2002}, Eremenko--Gabrielov computed the topological Brouwer degree of the \emph{real} Wronski map, working over a Schubert cell, with its standard orientation. Since this number constitutes a lower bound for the number of real points in any fiber, one might hope to prove that the fibers are entirely real (i.e. the result of \cite{MTV} above) by establishing a sufficiently large lower bound. Unfortunately, the real degree turns out to be zero when $m+d$ is even, and is equal to a count of shifted standard Young tableaux of a certain shape when $m+d$ is odd -- this count is strictly smaller than the complex degree. Brazelton \cite{brazeltonEnrichedDegreeWronski2025} gave a partial quadratic enrichment of this result, computing the $\A^1$-degree of the Wronski map over a Schubert cell, and giving a formula for its local degrees under mild characteristic assumptions.

Recently, Levinson--Purbhoo showed that an approach to the Shapiro--Shapiro conjecture via topological degrees can be made to work by using a different choice of orientation \cite{LevinsonPurbhoo}. Specifically, one divides the real projective space into \emph{chambers} $P_\mu \subset \R\P^{dm}$, where $\mu = (2^s 1^r)$ is a partition and $P_\mu$ is the set of polynomials with $r$ real and $2s$ complex conjugate roots. One then computes the degree of the real Wronski map (over each chamber) with respect to a different orientation called the \emph{character orientation}. This new degree turns out to be the symmetric group character value $\chi^\rect(\mu)$, where $\rect$ denotes the $d\times m$ rectangle. When $\mu = (1^{dm})$, this bound exactly implies the Shapiro--Shapiro conjecture, while for other choices of $\mu$ it yields new information about Wronskians with mixed real and complex roots. A noteworthy feature of this construction is that, although it involves real topological degrees and real orientability, the character orientation is, at its core, algebraic: it is given by multiplying the ordinary orientation by (the sign of) the equation for a certain divisor, the \emph{domino variety} $\Zdomino \subset \Gr(d, d+m)$, one of two divisors that lie over the boundary between the chambers $P_\mu$.

This paper, specifically the idea of twisting an orientation by the equation of a divisor, emerged from seeking to combine the approaches of Levinson--Purbhoo and Brazelton to define analogous ``twisted'' degrees of the Wronski map over arbitrary fields. We prove, in the language introduced in the first part of the paper, that the Wronski map admits an orientation $\rhodomino$ relative to the domino variety $\Zdomino$ (\Cref{lem:wronski-relatively-oriented-relative-to-vertdomino}). We then argue that certain \emph{twisted root spaces} of Wronski polynomials yield a complete class of $\A^1$-chambers for the Wronski map relative to the domino variety (see \Cref{def:sigma-action} and \Cref{lem:sym-is-domino-chamber}). In particular, every polynomial $w(t) \in k[t]_{\leq dm}$ is in one of these chambers; which chamber it is in is determined by the Galois orbit type of the roots of $w(t)$.

We thereby deduce that there is a well-defined twisted $\A^1$-degree of the Wronski map, twisted by the equation $\sdomino$ of $\Zdomino$, whose value is independent of the choice of fiber, depending only upon the Galois orbit type of the roots of the Wronskian (and on the base field). Our result holds for all fields of characteristic $0$ or in sufficiently large positive characteristic.

\begin{theorem}[{See \Cref{thm:main}}] Let $k$ be a field of characteristic $0$ or $\ge md$. Let $w(t) \in k[t]_{\leq dm}$ be a polynomial, corresponding to a point $w \in \P^{dm}(k)$. Then the twisted degree of the Wronski map over $w$, $\deg^{\sdomino}(\Wr,\rhodomino)(w)$, depends only on the Galois orbit type of the roots of $w(t)$. 

More precisely, for every continuous Galois action $\sigma \colon \Gal(k^s/k) \to \Sigma_{dm}$, there is a well-defined \emph{twisted degree} of the Wronski map $\deg^\sigma(\Wr) \in \GW(k)$, such that if the roots of $w(t)$ have the orbit type of $\sigma$, then $\deg^{\sdomino}(\Wr,\rhodomino)(w) = \deg^\sigma(\Wr)$.
\end{theorem}

Over the real numbers, the sign of $\deg^\sigma(\Wr)$ recovers the Brouwer degree of the Wronski map computed with respect to the relative character orientation in \cite{LevinsonPurbhoo}. Over other fields, however, there are more Galois actions on the roots which are not visible over the reals. For example we have the following (see \Cref{prop:Gr24} and ensuing discussion):

\begin{example} Over any field $k$ of characteristic $\ne 2,3$, the twisted degree of the Wronski map for $d=m=2$ computed with the character orientation $\rhodomino$ is
\[
    \deg^{\sdomino}(\Wr,\rhodomino) = \langle 3,f_{\Delta}\rangle \text{ in } \sheafGW(\mathbb{P}^4\setminus \Delta),
\]
where $f_\Delta$ is the standard equation for the discriminant of the Wronski polynomial. The square class of $\Delta$ depends upon the Wronski polynomial over which we carry out this computation, but only depends upon the Galois orbits of its roots by our main theorem. Over a finite field $\F_q$ of characteristic $\ge5$ for instance, the data of a Galois action on the four roots is the same as a partition of the roots, so we obtain the following twisted degrees:
\begin{center}
    \begin{tabular}{|r | c c c c c|}
        \hline
        $\sigma$ & $1^4$ & $2^1 1^2$ & $2^2$ & $3^1 1^1$ & $4^1$ \\
        \hline
        $\deg^{\sigma}(\Wr)$ & $\langle 1,3\rangle$ & $\langle 1,3u\rangle$ & $\langle 1,3\rangle$ & $\langle 1,3\rangle$ & $\langle 1,3u \rangle$\\
        \hline
    \end{tabular}
\end{center}
where $u\in \mathbb{F}_q$ is any non-square class.
\end{example}

It would be interesting to know $\deg^\sigma(\Wr)$ for general $d, m$ and $\sigma$, over an arbitrary field $k$.

\subsection{Acknowledgements} We would like to thank Kirsten Wickelgren for helpful conversations about this work. The first author would like to thank Nidhi Gupta and Anand Sawant for discussions about $\A^1$-connectivity. The second author thanks Nic Ford for a helpful conversation related to \Cref{ex:twisted-coordinate-ring}. Discussions with Claude were helpful for generalizing our lemmas on the basic properties of Wronski maps and domino varieties to positive characteristic, in particular \Cref{prop:Wronski-well-defined} and the last step of \Cref{prop:kempf-laksov-degree-calculation}. Claude also directed us to Laksov's work on the Wronski map in arbitrary characteristic using Hasse--Schmidt derivatives \cite{Laksov}. TB was partially supported by NSF Grant DMS-2303242. JL was partially supported by NSERC DG RGPIN-2021-04169.

\subsection{Code availability} The computations in \Cref{sec:examples} were done with the assistance of Mathematica and SageMath. The authors are happy to share this code upon request.

%% file: sections/background_a1.tex
\section{Background on \texorpdfstring{$\A^1$}{A1}-enumerative geometry}

Here we recap the theory of Grothendieck--Witt sheaves of non-degenerate symmetric bilinear forms on schemes, and how to produce $\A^1$-Brouwer degrees of finite flat lci oriented morphisms.

\subsection{The Grothendieck--Witt sheaf}

We recall the \emph{Grothendieck--Witt} sheaf of a qcqs scheme, first defined by Morel for smooth varieties over a field \cite{Morel}, which encodes isomorphism classes of families of quadratic forms.

\begin{notation} \label{notation:kxy}
Throughout, $k$ denotes a field and $X, Y$ are qcqs schemes over $k$.
\end{notation}

Consider the groupoid whose objects are pairs $(V,\beta)$ where $V$ is a locally free sheaf of $\mathcal{O}_Y$-modules of finite rank, and 
\begin{equation} \label{eqn:bilinear-form}
\beta \colon V\times V \to \mathcal{O}_Y
\end{equation}
is a nondegenerate symmetric bilinear form. This data defines a Nisnevich stack on qcqs schemes (in fact on qcqs algebraic spaces) by work of many authors \cite{Morel,Schlichting-exact,HoyoisLand}. We take the groupoid completion with respect to direct sum, obtaining a Nisnevich sheaf of abelian groups
\[
\sheafGW \colon \Sch_{\text{qcqs}}^\op \to \Ab,
\]
called the Grothendieck--Witt sheaf. In particular, the pairing $\beta$ determines a section $\beta \in \sheafGW(Y)$, and direct sums $(V \oplus V', \beta \oplus \beta')$ correspond to sums $\beta+\beta' \in \sheafGW(Y)$. Conversely, any section can be represented locally by a formal sum of nondegenerate symmetric bilinear forms as in \Cref{eqn:bilinear-form} on open subsets $U \subset Y$. The Grothendieck--Witt sheaf is unramified in the sense of \cite{Ojanguren-Panin}, meaning it is insensitive to the deletion of a subspace of codimension $\ge 2$ when the ambient scheme is smooth: if $Y$ is smooth over $k$ and $Z \subset Y$ has codimension $\geq 2$, then the restriction map
\[\sheafGW(Y) \to \sheafGW(Y \setminus Z)\]
is an isomorphism. 

The value of $\sheafGW$ on the spectrum of a field recovers the classical Grothendieck--Witt ring:

\begin{definition}\label{def:GW} We denote by $\GW(k):= \sheafGW(\Spec(k))$ the \emph{Grothendieck--Witt ring} of the field $k$. It is isomorphic to the group completion of the semiring of isomorphism classes of non-degenerate symmetric bilinear forms over $k$.
\end{definition}

\begin{notation} Given a unit $a \in k^\times$, we denote by $\langle a \rangle$ the bilinear form defined by
\begin{align*}
    \langle a \rangle \colon k \times k &\to k \\
(x,y) &\mapsto axy
\end{align*}
These form additive generators for $\GW(k)$. We note that $\langle ab^2 \rangle = \langle a \rangle$ whenever $a,b \in k^\times$, so that $\GW(k)$ is in fact generated by the classes of the elements of $k^\times / k^{\times 2}$. We will make use of the following notation for a sum of elements, and for the \emph{hyperbolic form}, respectively:
\begin{align*}
    \langle a_1,\ldots,a_n\rangle &:= \langle a_1 \rangle + \ldots \langle a_n \rangle \\
    h &:= \langle 1,-1 \rangle.
\end{align*}
\end{notation}
\begin{remark}
When $2$ is inverted over a field (or more generally over a suitably nice ring, see \cite{milnorSymmetricBilinearForms1973}), the Grothendieck--Witt ring captures the theory of quadratic forms over the ring. We refer the reader to \cite{Lam-quad} for more background on this and for general facts about the Grothendieck--Witt ring of a field.
\end{remark}

\begin{notation}\label{notation:restriction-on-GW}
Let $L/k$ be a field extension. The map $p\colon\Spec(L) \to \Spec(k)$ induces a natural homomorphism
\[
p^\ast \colon \GW(k) \to \GW(L).
\]
Explicitly, given the form $\beta : V \times V \to k$, we tensor up to $L$ to obtain a form $(V\otimes_k L) \times (V \otimes_k L) \to L$.
\end{notation}

\begin{notation} Given any invertible $\mathcal{O}_X$-module $\mathcal{L}$, we can define the twisted Grothendieck--Witt sheaf $\sheafGW(X,\mathcal{L})$ from isomorphism classes of $\mathcal{L}$-valued non-degenerate symmetric bilinear forms
\[
V\times V \to \mathcal{L}.
\]
The tensor product of forms induces a natural pairing for any line bundles $\mathcal{L}_i$:
\[
\sheafGW(X,\mathcal{L}_1) \times \sheafGW(X,\mathcal{L}_2) \to \sheafGW(X,\mathcal{L}_1 \otimes \mathcal{L}_2).
\]
\end{notation}

\begin{remark}\label{rmk:GW-twist-by-squares}
For any line bundles $\mathcal{L}_1$ and $\mathcal{L}_2$, twisting $V$ by $\mathcal{L}_2$ yields a \emph{canonical} isomorphism $\sheafGW(X,\mathcal{L}_1) \cong \sheafGW(X,\mathcal{L}_1\otimes \mathcal{L}_2^{\otimes 2})$. Hence we might talk about equality of two forms, where one form lives in one group and one in another, and this is well-defined. As a subtlety to be aware of: if we know only that two line bundles $\mathcal{L}$ and $\mathcal{L}'$ define the same class in $\Pic(X)/2$ then $\sheafGW(X,\mathcal{L})$ and $\sheafGW(X,\mathcal{L}')$ are isomorphic, but \emph{not canonically so}. To exhibit an isomorphism one must pick a square root line bundle (i.e. a choice of $\mathcal{L}_2$ above) -- this is called an \emph{alignment} in \cite{BalmerCalmes}.
\end{remark}

\subsection{Duality and transfer} A morphism of schemes $f \colon X \to Y$ induces a pullback on global sections $\sheafGW(Y) \to \sheafGW(X)$. We now recall how, for a sufficiently nice morphism, we can push forward twisted forms from $X$ back to $Y$.

\begin{notation}\label{notation:f}
For the remainder of this section, continuing \Cref{notation:kxy}, $f : X \to Y$ denotes a finite flat lci (locally complete intersection) morphism.
\end{notation}

Under the assumptions of \Cref{notation:f}, we can construct a line bundle $\omega_f$ on $X$ as the determinant of the cotangent complex of the morphism $f$. Grothendieck-Serre duality supplies us with a well-defined \emph{trace} between the proper pushforward of this complex and the structure sheaf on $Y$ (see e.g. \cite[Appendix~B]{bachmannEulerClassesSixfunctors2023}):
\begin{equation}\label{eqn:trace}
    f_\ast \omega_f \xto{\Tr} \mathcal{O}_Y.
\end{equation}
Given any symmetric bilinear form $V \times V \to \omega_f$ over $X$, we can apply $f_\ast$ and post-compose with the trace to obtain a symmetric bilinear form over $Y$ (note that $f_\ast V$ is still a locally free sheaf on $Y$ since $f$ is finite and flat). This operation preserves non-degenerate forms (c.f.~\cite[2.5]{KLSW-count}), therefore defines a pushforward map
\begin{equation}\label{eqn:GW-pushforward}
    f_\ast \colon \sheafGW(X,\omega_f) \to \sheafGW(Y).
\end{equation}

\begin{example} \label{ex:trace}
Let $L/k$ be a finite separable field extension. The associated cotangent complex is trivial, so we obtain the field trace $\Tr_{L/k} : \GW(L) \to \GW(k)$.
\end{example}

\subsection{Orientations and degrees}

We now recall the definition of $\mathbb{A}^1$-degree.

\begin{terminology}\label{term:oriented-orientation-orientation-class}
    Following \cite[2.2]{KLSW-count}\footnote{Our convention for $\rho$ is inverse from theirs, but the definition is essentially the same.}, we define an \emph{orientation} of $f$ to be a choice of line bundle $\mathcal{L} \to X$ and an isomorphism $\rho \colon \mathcal{L}^{\otimes 2} \xto{\sim} \omega_f$. We say $f$ is \emph{orientable} if an orientation exists, i.e. if $\omega_f$ is isomorphic to the square of a line bundle, and \emph{oriented} if we have fixed an orientation. We may likewise think of $\rho$ as a pairing $\mathcal{L} \times \mathcal{L} \to \omega_f$; we call the corresponding section $\rho \in \sheafGW(X, \omega_f)$ an \emph{orientation class} (of $f$).
\end{terminology}

 Putting this all together, we have:

\begin{definition}[{\cite[2.4]{KLSW-count}}]
\label{def:degree}
Let $\rho$ be an orientation of $f$. We define the ($\mathbb{A}^1$-)\emph{degree} of $f$ to be the class
\begin{align*}
    \deg(f,\rho) := f_\ast \rho \in \sheafGW(Y),
\end{align*}
where $f_\ast$ is the pushforward from \Cref{eqn:GW-pushforward} and $\rho$ is the orientation class from \Cref{term:oriented-orientation-orientation-class}.
\end{definition}

This degree resembles the classical Brouwer degree in the following sense. Given a closed point $y\colon \Spec(k(y)) \to Y$, we can pull back $\deg(f,\rho) \in \sheafGW(Y)$ to obtain a class we denote by $\deg(f,\rho)(y)$ in $\GW(k(y))$, which is a quadratic form. If $f$ is \'etale over $y$, then the degree is obtained by counting points along the fiber, weighted by the square class of the Jacobian of $f$ at $x$, as in the following proposition. This is analogous to the classical Brouwer degree, where we count points weighted by the sign of the Jacobian (over $\mathbb{R}$).
\begin{proposition}[{\cite[3.10]{KLSW-count}}]\label{prop:sum-of-local-degrees}
Let $\rho$ be an orientation of $f$. Then for any $y\in Y$ over which $f$ is \'etale, we have an equality in $\GW(k(y))$ of the form
\begin{equation} \label{eqn:a1deg-jacobian}
    \deg(f,\rho)(y) = \sum_{x\in f^{-1}(y)}\Tr_{k(x)/k(y)} \langle \Jac(f)(x) \rangle.
\end{equation}
The summand for $x$ is called the \emph{local degree} of $f$ at $x$ (with respect to $\rho$). Note that the Jacobian is computed in charts compatible with the orientation $\rho$ (for further details, see \cite{KLSW-count}).
\end{proposition}

We note that $\mathbb{A}^1$-degrees pull back along flat maps:

\begin{proposition}[{\cite[2.8]{KLSW-count}}]\label{prop:a1deg-flat-pullback}
\label{prop:pullback-of-deg} Let $\rho$ be an orientation of $f$. Let $g\colon Y' \to Y$ be flat and let $f'$ and $g'$ be defined by the pullback diagram
\[\begin{tikzcd}
    X'\rar["g'"]\dar["f'" left] & X\dar["f" right] \\
    Y'\rar["g" below] & Y.
\end{tikzcd}\]
Then $f'$ is finite flat lci and oriented by $g'^*\rho$, and $\deg(f',g^{\prime\ast}\rho) = g^\ast \deg(f,\rho)$ in $\sheafGW(Y')$.
\end{proposition}

\subsubsection{Choices of orientation}

It is common in the literature to suppress $\rho$ and simply write $\deg(f)$ when the orientation is implicitly understood. It is important, though, to note that degree does depend intrinsically on the choice of orientation $\rho$, as the following example illuminates:

\begin{example}\label{exa:rho-matters}
Consider the identity map $\id \colon \Spec(k) \to \Spec(k)$. Its cotangent complex is trivial of rank one, as is any line bundle, so a choice of isomorphism $\rho$ is a choice of multiplication by a scalar $u\in k^\times$. For such a $\rho$, call it $\rho_u$ to remember the scalar, the degree $\deg(f,\rho_u) \in \GW(k)$ is the composite
\begin{align*}
    k \times k & \to k \\
    (x,y) &\mapsto uxy,
\end{align*}
which is the class $\langle u\rangle$. This clearly can change depending upon the square class of $u$.
\end{example}

In special cases, the choice of orientation can safely be suppressed, as it does not severely affect the resulting computation:

\begin{example}\label{ex:proper-geometrically-connected-deg-independent-of-orientation}
Suppose $X$ has $H^0(\mathcal{O}_X^\times) = k^\times$ (e.g. $X$ proper and geometrically connected) and $\Pic(X)[2]=0$. Then any two orientations $\rho_1 : \mathcal{L}_1^{\otimes 2} \to \omega_f$ and $\rho_2 : \mathcal{L}_2^{\otimes 2} \to \omega_f$ of $f$ are related by a scalar $u \in k^\times$, and in particular
\begin{align*}
    \deg(f,\rho_1) = \langle u \rangle \deg(f,\rho_2).
\end{align*}
To see this: since $\Pic(X)[2] = 0$, any two square roots of $\omega_f$ are isomorphic, so there is an isomorphism $\phi \colon \mathcal{L}_1 \xto{\sim} \mathcal{L}_2$. The composite
\begin{align*}
    \omega_f \xto{\rho_1^{-1}} \mathcal{L}_1^{\otimes 2} \xto{\phi^{\otimes 2}}\mathcal{L}_2^{\otimes 2} \xto{\rho_2}\omega_f
\end{align*}
is an automorphism of $\omega_f$, hence is multiplication by a scalar $u \in H^0(\mathcal{O}_X^\times) = k^\times$. Thus the two degrees differ by multiplication by the class $\langle u\rangle$.
\end{example}

However, in general and in our application to Wronski maps, different orientations can change the degree by much more than multiplication by a scalar.

\begin{example} Let $X=Y = \mathbb{G}_m = \Spec(k[t,t^{-1}])$, and let $f\colon X \to Y$ be the squaring map sending $t\mapsto t^2$. The square classes of global units are $k^\times / (k^\times)^2 \times \Z/2$, where the extra cyclic group of order two comes from whether the power on $t$ is even or odd. We compute the degree using \Cref{prop:sum-of-local-degrees} with $y = 1$. Letting $\rho \colon \mathcal{O}_X \cong \mathcal{O}_X^{\otimes 2}$ send $1 \mapsto 1 \otimes 1$, the local degree is $\langle 1 \rangle$ at $x=1$ and is $\langle -1 \rangle$ at $x = -1$, so
\[\deg(f,\rho)(1) = \langle 1,-1\rangle\]
is the hyperbolic form. If we instead take $\rho'$ to be the isomorphism sending $1\mapsto 1\otimes t$, the local degree at $x = 1$ does not change, while the local degree at $x = -1$ changes from to $\langle 1\rangle$, so
\[\deg(f, \rho')(1) = 2 \langle 1 \rangle.\]
Indeed, as sections of $\sheafGW(Y)$ we have $\deg(f, \rho) = \langle 1, -1 \rangle$ and $\deg(f,\rho') = \langle 1,t\rangle \in \sheafGW(\mathbb{G}_m)$. We may also observe that $\deg(f, \rho')(y)$ depends on $y$.
\end{example}

\subsection{\texorpdfstring{$\mathbb{A}^1$}{A1}-connectedness} Given a class in $\sheafGW(Y)$ (e.g. a degree, as above), we can take a closed point $y\in Y$ and pull back along the inclusion $\Spec(k(y)) \to Y$ to obtain a class in $\GW(k(y))$. One can ask: when is this form independent of the choice of closed point?

Harder's theorem \cite[VII]{Lam-SerresProblem} states that non-degenerate symmetric bilinear forms over $k[t]$ are always extended from forms over $k$. This leads to the observation that if two points in $Y$ with the same residue field $L$ can be connected along an affine line $\A^1_L$, then the pullback of the form in $\sheafGW(Y)$ to either of the points will yield the same class in $\GW(L)$. To that end, we recall the following notion from $\A^1$-homotopy theory:

\begin{definition}[{\cite{AsokMorel}}]
Let $Y$ be a smooth $k$-variety. We say that $Y$ is \emph{$\A^1$-chain connected} if, for every finitely generated separable extension $L/k$ and every pair of points $x,y\in Y(L)$, there is some $n$ and some sequence of points $x=x_0,x_1,\ldots, x_n=y$ and morphisms
\[
    g_i \colon \A^1_L \to Y_L
\]
so that $g_i(0) = x_i$ and $g_i(1) = x_{i+1}$ for all $i\ge 1$.
\end{definition}

\begin{proposition}\label{prop:a1connected-constancy} If $Y$ is an $\A^1$-chain connected $k$-variety, then we have a canonical isomorphism $\GW(k) \xto{\sim} \sheafGW(Y)$, induced by the structure morphism $Y \to \Spec(k)$. In particular, for any $\beta \in \sheafGW(Y)$, there is a unique class $\overline{\beta}\in \GW(k)$, such that:
\begin{enumerate}
    \item For any finite field extension $L/k$ and any $L$-point $y\in Y$, we have that $\beta(y) = p^\ast \overline{\beta}$ in $\GW(L)$, where $p$ is the pullback (\Cref{notation:restriction-on-GW}).
    \item If $\rho$ is an orientation of $f\colon X \to Y$, then the sum of local degrees in \Cref{eqn:a1deg-jacobian} is always equal to $p^*\overline{\beta}$. In particular, it is independent of the choice of fiber.
\end{enumerate}
As a special case of (1), for any $k$-point $y\in Y$, we have $\beta(y) = \overline{\beta}$ in $\GW(k)$.
\end{proposition}

There is a weaker notion of a variety being $\mathbb{A}^1$-\emph{connected}, which is sufficient to extract a class in $\GW(k)$ from a class in $\sheafGW(Y)$ (see e.g. \cite[2.35]{KLSW-count}). Every $\A^1$-chain connected variety is $\A^1$-connected, but the converse doesn't hold in general \cite{AsokMorel}. For the purposes of this paper, we only need to work with $\A^1$-chain connectivity to state our results. The following slight weakening will also be useful:

\begin{definition}
We say $Y$ is \emph{essentially $\A^1$-connected} if there exists a smooth $\A^1$-connected variety $\overline{Y}$ and an open immersion $Y \hookto \overline{Y}$ such that $\overline{Y} \setminus Y$ has codimension $2$ in $\overline{Y}$.
\footnote{Essentially $\A^1$-connected varieties over infinite fields are expected to be themselves $\A^1$-connected, although this is unknown at the time of writing.}
\end{definition}

The analogue of \Cref{prop:a1connected-constancy} holds for essentially $\A^1$-connected varieties because
\[\sheafGW(\overline{Y}) \xto{\sim} \sheafGW(Y)\]
is an isomorphism by unramifiedness of $\sheafGW$. We also note that the property of being essentially $\A^1$-connected is preserved under deletion of loci of codimension $\geq 2$, whereas $\A^1$-chain connectedness need not be preserved \cite[2.16]{AD-excision}.

\begin{example}[Examples of $\A^1$-connected $k$-schemes]\label{ex:a1-connected-schemes} Projective space, Grassmannians, generalized flag varieties, smooth proper toric varieties, and the moduli space $\overline{M}_{0, n}$ are all $\A^1$-chain connected.
\end{example}

We have the following example of a smooth variety which fails to be $\A^1$-chain connected.

\begin{example} The moduli space $M_{0,n}$ is not $\A^1$-chain connected for $n\ge4$.
\end{example}
\begin{proof}
For $n=4$, any morphism $g \colon \A^1 \to M_{0,4}$ extends to a map $\overline{g} \colon \P^1 \to \overline{M}_{0,4}$, since $\P^1$ is a smooth curve and $\overline{M}_{0,4}$ is proper. If $g$ is nonconstant, then $g$ is dominant, hence $\overline{g}$ is surjective, so at least three points on $\P^1$ map to the three boundary points of $\overline{M}_{0,4}$, hence at least two points on $\A^1$ map to the boundary, which is a contradiction. Thus any such $g$ is constant. For general $n \geq 4$, we use the natural embedding $\varphi \colon \overline{M}_{0, n} \hookrightarrow (\overline{M}_{0, 4})^{\binom{n}{4}}$ via the product of the all possible projection maps $\overline{M}_{0,n} \to \overline{M}_{0,4}$ remembering four marked points. Any map $g \colon \mathbb{A}^1 \to M_{0, n}$ extends to a map $\overline{g} \colon \P^1 \to \overline{M}_{0,n}$. By the observations above, $\varphi \circ \overline{g}$ is constant along each factor, hence $\varphi \circ \overline{g}$ is constant. Since $\varphi$ is an embedding, it follows that $\overline{g}$, and therefore $g$, are likewise constant.
\end{proof}

%% file: sections/twisted_degrees.tex
\section{Twisted degrees and \texorpdfstring{$\A^1$}{A1}-chambers}

We now define `twisted' $\A^1$-degrees for maps $f : X \to Y$ under weakened orientability and $\A^1$-connectivity hypotheses. We state versions of \Cref{term:oriented-orientation-orientation-class} and \Cref{def:degree} (see \Cref{def:orientation-relative-to-divisor} and \Cref{def:twisted-degree}), the Jacobian formula \Cref{prop:sum-of-local-degrees} (see \Cref{prop:twisted-degree-sum-of-local-degrees}), flat pullback \Cref{prop:pullback-of-deg} (\Cref{prop:pullback-of-twisted-deg}) and fiber invariance \Cref{prop:a1connected-constancy} (see \Cref{prop:degs-same-across-a1-chamber} and \Cref{thm:fiber-invariance-twisted-degrees-on-chambers}).

Before these results, we first discuss two closely-related ways of twisting Grothendieck--Witt classes: by units up to squares, and by sections of line bundles which are associated to divisors. We then define twisted $\A^1$-degrees for maps $f$ that are not necessarily orientable in the sense of \Cref{term:oriented-orientation-orientation-class}. Finally, we introduce the notion of $\A^1$-chamber, for the case where $Y$ is not $\A^1$-connected (but is perhaps `close' to being $\A^1$-connected).

\subsection{Twisting Grothendieck--Witt classes by units up to squares}
Here we discuss how Grothendieck--Witt classes can be twisted by sections of the sheaf of invertible regular functions up to squares. Let $s\in H^0(X,\mathcal{O}_X^\times)$ and let $\beta \in \sheafGW(X,L)$ be a class represented by a form $\beta \colon V \times V \to L$. We define $s\cdot \beta$ to be the form
\begin{equation} \label{eqn:twisted-form}
V\times V \xto{\beta} L \xto{s} L.
\end{equation}
This defines an action of $H^0(X,\mathcal{O}_X^\times)$ on $\sheafGW(X,L)$. It is clear that sections of $\mathcal{O}_X^{\times 2}$ act trivially, so this gives an action of $\frac{H^0(\mathcal{O}_X^\times)}{H^0(\mathcal{O}_X^{\times2})}$ on $\sheafGW(X,\mathcal{L})$. It in fact extends to an action of $H^0(\mathcal{O}_X^\times / \mathcal{O}_X^{\times 2})$:

\begin{proposition}\label{prop:square-class-action}
Let $\mathcal{L} \to X$ be any line bundle. Then there is a well-defined multiplication action of $H^0(\mathcal{O}_X^\times / \mathcal{O}_X^{\times 2})$ on $\sheafGW(X,\mathcal{L})$.
\end{proposition}
\begin{proof}
Let $s \in H^0(\mathcal{O}_X^\times / \mathcal{O}_X^{\times 2})$ be a global section, given by data $(U_i, s_i)$, where $s_i$ is a unit and $s_i / s_j$ is a square on $U_{ij} := U_i \cap U_j$. Let
$V \times V \xto{\beta} \mathcal{L}$ be as above. On each $U_i$, we twist by $s_i$ as in \Cref{eqn:twisted-form} to obtain twisted pairings $\tilde \beta_i := s_i \cdot \beta|_{U_i} \in \sheafGW(U_i, \mathcal{L}|_{U_i})$. On double overlaps $U_{ij}$,
\[\tilde \beta_i|_{U_{ij}} = \tilde \beta_j|_{U_{ij}} \in \sheafGW(U_{ij}, \mathcal{L}|_{U_{ij}})\]
since $s_i/s_j$ is a square on $U_{ij}$. Thus these classes glue to a section of $\sheafGW(X, \mathcal{L})$.
\end{proof}

\begin{notation} In \Cref{prop:square-class-action}, if $s\in H^0(\mathcal{O}_X^\times / \mathcal{O}_X^{\times 2})$ and $\beta \in \sheafGW(X,\mathcal{L})$, we denote the resulting scaled form by $s\cdot \beta \in \sheafGW(X,\mathcal{L})$.
\end{notation}

We note that not all sections of $\mathcal{O}_X^\times / \mathcal{O}_X^{\times 2}$ come from global invertible regular functions on $X$. In nice cases, we can characterize the sections precisely. Although we do not make use of the following result in the proof of our main theorem, we include it for the curious reader:
 
\begin{proposition}\label{prop:SES-global-units-mod-squares}
Let $X$ be a disjoint union of integral schemes (e.g. $X$ smooth). Then there is a short exact sequence of the form
\begin{align}\label{eqn:ses-for-twisting}
    0 \to \frac{H^0(\mathcal{O}_X^\times)}{H^0(\mathcal{O}_X^\times)^2} \to H^0(\mathcal{O}_X^\times / \mathcal{O}_X^{\times 2}) \to \Pic(X)[2] \to 0
\end{align}
\end{proposition}
\begin{proof} The proof reduces to the case where $X$ is integral. Consider the exact sequence of Zariski sheaves
\begin{align*}
    (*) \quad 0 \to \mu_2 \to \mathcal{O}_X^\times \xrightarrow{\mathrm{sq}} \mathcal{O}_X^\times \to \mathcal{O}_X^\times / \mathcal{O}_X^{\times 2} \to 0,
\end{align*}
where $\mathrm{sq}$ is the squaring map. The sheaf $\mu_2$ is flasque, so taking hypercohomology yields a spectral sequence with $E_1$ page as follows, showing only the first three rows:
\[
\begin{tikzcd}
  H^0 : & \{\pm 1\} \arrow[r] & H^0(\mathcal{O}_X^\times) \arrow[r, "\mathrm{sq}"] & H^0(\mathcal{O}_X^\times) \arrow[r] & H^0(\mathcal{O}_X^\times / \mathcal{O}_X^{\times 2}) \\
  H^1 : & 0  & \Pic(X) \arrow[r, "\mathrm{sq}"] \arrow[rru, dashed] & \Pic(X) \arrow[r] & H^1(\mathcal{O}_X^\times / \mathcal{O}_X^{\times 2}) \\
  H^2 : & 0  & 0  &0  & H^2(\mathcal{O}_X^\times / \mathcal{O}_X^{\times 2})
\end{tikzcd}
\]
The dashed arrow comes from the $E_2$ page. The sequence abuts to $0$ since $(*)$ is exact, so we obtain an isomorphism
\[\coker\bigg(\frac{H^0(\mathcal{O}_X^\times)}{H^0(\mathcal{O}_X^\times)^2} \to H^0(\mathcal{O}_X^\times / \mathcal{O}_X^{\times 2}) \bigg) \cong \ker\big( \Pic(X) \xrightarrow{\mathrm{sq}} \Pic(X) \big) = \Pic(X)[2] .\]
If $X$ is a disjoint union of integral (smooth) schemes, the calculation above distributes over the components of $X$.
\end{proof}
The map to the 2-torsion elements in the Picard group in \Cref{prop:SES-global-units-mod-squares} is given as follows: let $s$ be represented by $(U_i, s_i)$, where each $s_i$ is invertible on $U_i$ and $s_i / s_j$ is a square. Let $g_{ij}$ be a choice of square root of $s_i/s_j$. We define on triple overlaps $U_{ijk}$,
\[\epsilon_{ijk} := g_{ij}g_{jk}g_{ki} \in \{\pm 1\},\]
observing that $\epsilon_{ijk}^2 = 1$ by definition of $s$. Since $H^2_\Zar(X,\mu_2) = 0$, we can globally resolve the signs on the $\epsilon_{ijk}$ so that they are all $1$. Then the $g_{ij}$ yield a line bundle, which squares to the trivial bundle since $g_{ij}^2 = s_i / s_j$.

\begin{remark}
If $X$ is reducible, or more generally if $H^1(X, \mu_2) \ne 0$, the left term in \Cref{eqn:ses-for-twisting} should be replaced by $\frac{H^0(\mathcal{O}_X^\times)}{H^0(\mathcal{O}_X^{\times2})}$ to account for local squares. If $H^2(X, \mu_2) \ne 0$, the description is more complicated and we do not get a well-defined map to $\Pic(X)[2]$. 
\end{remark}

\subsection{Twisting Grothendieck--Witt classes by sections of divisor sheaves} There is a more general notion of twisting we can do, which is by a section of $\mathcal{O}(D)$, where $D$ is a Cartier divisor. The twists by sections of $\mathcal{O}_X^\times$ are the particular case where the divisor $D$ is the zero divisor. However in this more general setting, the resulting classes are only non-degenerate over the complement of $D$.

\begin{terminology}
We will say an \emph{equation for $D$} is a section $s \in H^0(\mathcal{O}(D))$ that locally generates the ideal sheaf of $D$.
\end{terminology}

\begin{notation}\label{nota:OD-section-times-form}
Let $\beta \colon V \times V \to \mathcal{L}$ a non-degenerate symmetric bilinear form on $X$. Let $D$ be a Cartier divisor and $s$ an equation for $D$. The composite
\[
V \times V \xto{\beta} \mathcal{L} \xto{s} \mathcal{L}(D)
\]
is a symmetric bilinear form on $X$, which visibly vanishes along $D$. Pulling it back to $X\setminus D$ yields a well-defined class in $\sheafGW(X \setminus D, \mathcal{L}(D))$, which we denote $s \cdot \beta$.
\end{notation}
Just as in \Cref{prop:square-class-action}, we have a version of invariance of this twisting operation under squares:

\begin{proposition}
\label{prop:invariance-twisted-forms-under-squares}
Let $\beta \in \sheafGW(X, \mathcal{L})$, let $D,E$ be Cartier divisors on $X$, and let $s$ and $t$ be equations for $D$ and $E$, respectively. We have twisted forms $s\cdot \beta \in \sheafGW(X \setminus D,\mathcal{L}(D))$ and $st^2 \cdot \beta\in \sheafGW(X\setminus (D\cup E), \mathcal{L}(D+2E))$. Then after pulling back $s\cdot \beta$ to $X\setminus(D\cup E)$, we have that $s\cdot \beta$ and $st^2 \cdot \beta$ are equal.\footnote{Equality here is as in \Cref{rmk:GW-twist-by-squares}: the forms are identified along the canonical isomorphism $\sheafGW(X \setminus(D\cup E), \mathcal{L}(D)) \cong \sheafGW(X\setminus (D\cup E), \mathcal{L}(D+2E))$.}
\end{proposition}
\begin{proof} Let $\beta$ be represented by $V \times V \to \mathcal{L}$. Then this proposition comes from the observation that the diagram commutes:
\[\begin{tikzcd}
    V(-E)\times V(-E)\rar["\beta"]\dar["t\times t" left] & \mathcal{L}(-2E)\rar["st^2"]\dar["t^2"] & \mathcal{L}(D)\dar[equal] \\
    V\times V\rar["\beta" below] & \mathcal{L}\rar["s" below] & L(D).
\end{tikzcd} \qedhere\]
\end{proof}

As a special case, \Cref{prop:invariance-twisted-forms-under-squares} allows us to remove vanishing loci of even order, extending the forms back over all of $X$:

\begin{remark}\label{rmk:twisting-by-a-section-of-2D}
Let $D$ be a Cartier divisor. Let $s$ be an equation for $2D$. Then for any equation $t$ for $D$, we may write $s = ut^2$ for some $u\in H^0(\mathcal{O}^\times)$. By \Cref{prop:invariance-twisted-forms-under-squares}, for any $\beta \in \sheafGW(X,\mathcal{L})$,
\[s \cdot \beta = u \cdot \beta |_{X \setminus D}.\]
The class $u \cdot \beta \in \sheafGW(X)$ is moreover independent of the choice of $t$. Indeed it is even harmless to vary the choice of $t$ (and therefore $u$) locally on $X$, since the class of $u = s/t^2$ in $\mathcal{O}_X^\times/\mathcal{O}_X^{\times 2}$ does not change. Note that $u \cdot \beta$ need not equal $\beta$.
\end{remark}

\subsection{Twisted orientations and twisted degrees}

We now introduce what we call the \emph{twisted degree} of a map relative to a divisor. We are motivated both by the idea of orientations of vector bundles relative to divisors, following \cite{LarsonVogt}, as well as the theory of scaled trace forms, whose use in real root counting is classically known (see \cite{RealRoots} for a nice introduction). 

\begin{notation}\label{notation:kxyf}
For the remainder of this section, $k$ is a field, $X$ and $Y$ are qcqs $k$-schemes, $f : X \to Y$ is finite flat and lci, $D$ is a Cartier divisor on $X$ and $s$ is an equation for $D$.
\end{notation}

We make the following natural definition.

\begin{definition}\label{def:orientation-relative-to-divisor}
A \emph{$D$-orientation of $f$} is a line bundle $\mathcal{L} \to X$ and a choice of isomorphism
\begin{align*}
    \rho \colon \mathcal{L}^{\otimes 2} \xto{\sim} \omega_f \otimes \mathcal{O}(-D).
\end{align*}
We say $f$ is \emph{$D$-orientable} if it admits such an isomorphism and \emph{$D$-oriented} if we have fixed one, and we call the resulting class in $\sheafGW(X, \omega_f(-D))$ a \emph{$D$-orientation class} (of $f$). Finally, we call $s\cdot \rho \in \sheafGW(X \setminus D, \omega_f)$ a \emph{twisted orientation class} of $f$.
\end{definition}

\begin{definition}\label{def:twisted-degree}
We define the \emph{twisted degree} of $f$ (with respect to $D$ and $s$) as follows: we restrict the twisted orientation class $s \cdot \rho$ to $X \setminus f^{-1}f(D)$, then push forward along $X \setminus f^{-1}f(D) \to Y \setminus f(D)$ (which is again finite flat and lci). We denote it 
\[\deg^s(f,\rho) := f_*( s \cdot \rho|_{X \setminus f^{-1}f(D)}) \in \sheafGW(Y \setminus f(D)).\]
Explicitly, it is the restriction of the following symmetric bilinear form to $Y\setminus f(D)$:
\[
f_\ast \mathcal{L} \times f_\ast \mathcal{L} \to f_\ast (\mathcal{L}^{\otimes 2}) \xto{f_*\rho} f_\ast(\omega_f \otimes \mathcal{O}(-D)) \xto{f_\ast s} f_\ast \omega_f \xto{\Tr} \mathcal{O}_Y.
\]
\end{definition}

Similarly to \Cref{prop:sum-of-local-degrees}, we can pull back a twisted degree $\deg^s(f,\rho) \in \sheafGW(Y \setminus f(D))$ to a point $y\in \GW(k(y))$ to get a class we write as $\deg^s(f,\rho)(y)$. We then have:

\begin{proposition}[Twisted local degree formula]\label{prop:twisted-degree-sum-of-local-degrees}
Let $\rho$ be a $D$-orientation of $f$. Then for any point $y\in Y \setminus f(D)$ over which $f$ is \'etale, we have an equality in $\GW(k(y))$ of the form
\[
    \deg^s(f,\rho)(y) = \sum_{x\in f^{-1}(y)} \Tr_{k(x) / k(y)} \langle \Jac(f)(x)\cdot s(x)\rangle.
\]
Thus the local degree is twisted by the square class of the section $s$ across the fiber.
\end{proposition}
\begin{proof} When we pull $\deg^s(f,\rho) \in \sheafGW(Y\setminus f(D))$ back to $y$, it decomposes as a sum of forms over the preimages of $y$, as in \cite[Definition~3.1]{KLSW-count}. By keeping track of the scaling by $s$ in the proofs of \cite[3.5,3.8]{KLSW-count}, the result follows.
\end{proof}

As an analogy of \Cref{prop:pullback-of-deg} we have the following:
\begin{proposition}\label{prop:pullback-of-twisted-deg} Let $\rho$ be a $D$-orientation of $f$. Let $g\colon Y' \to Y$ be flat and let $f'$ and $g'$ be defined by the pullback diagram
\[\begin{tikzcd}
    X'\rar["g'"]\dar["f'" left] & X\dar["f" right] \\
    Y'\rar["g" below] & Y.
\end{tikzcd}\]
Then $f'$ is finite flat and lci, $g'^*\rho$ is a $g'^*D$-orientation of $f'$, and we have an equality in $\sheafGW(Y' \setminus g^{-1}f(D))$ of the form
\[
    g^\ast \deg^s(f,\rho) = \deg^{g^{\prime\ast} s}(f',g^{\prime\ast}\rho).
\]
\end{proposition}

\begin{remark}\label{rmk:twisting-doesnt-commute-with-pushforward}
Twisting the orientation class does not commute with pushforward and can change the degree of the map, just like changing the orientation, e.g. by altering the local degrees across a fiber. Even when $D = 0$ and $s \cdot \rho$ differs from $\rho$ by global invertible function on $X$, we should not expect $\deg^s(f,\rho)$ to differ from $\deg(f, \rho)$ in a simple way (unless $X$ has very few such functions).
\end{remark}

We note the behavior on twisted degrees, considered up to squares:

\begin{proposition}[Twisted degrees up to squares]
\label{prop:twisted-degrees-up-to-squares}
Let $\rho : \mathcal{L}^{\otimes 2} \xto{\sim} \omega_f \otimes \mathcal{O}(-D)$ be a $D$-orientation of $f$. Let $E$ be another Cartier divisor on $X$ and $t$ an equation for $E$. Let $\rho(-2E)$ be the twist of $\rho$ by $\mathcal{O}(-2E)$. Then $\deg^{s}(f,\rho) = \deg^{st^2}(f,\rho(-2E))$ in $\sheafGW(Y \setminus f(D \cup E))$.
\end{proposition}
\begin{proof} This follows by \Cref{prop:invariance-twisted-forms-under-squares}.
\end{proof}

A particular case is when $f$ is oriented relative to the square of a divisor. We can then extract two different degrees, which are twisted and untwisted:

\begin{remark}\label{rmk:twisting-degree-by-section-of-2D} Let $\rho : \mathcal{L}^{\otimes 2} \xto{\rho} \omega_f \otimes \mathcal{O}(-D)$ be a $D$-orientation of $f$ and assume that $D = 2E$, where $E$ is a Cartier divisor.
\begin{enumerate}
    \item The class $\rho(E)$ defined by
    \begin{equation}\label{eqn:rho-twisted-by-half-D}
        \mathcal{L}(E) \times \mathcal{L}(E) \to \mathcal{L}^{\otimes 2} \otimes \mathcal{O}(2E) \xto{\rho \otimes \id_{\mathcal{O}(2E)}} \omega_f,
    \end{equation}
    is an orientation class of $f$ in the sense of \Cref{term:oriented-orientation-orientation-class}. Thus we obtain a well-defined (untwisted) degree
    \[
    \deg(f,\rho(E)) = f_*(\rho(E)) \in \sheafGW(Y).
    \]
    \item By writing $s=ut^2$ for some equation $t$ for $E$ and some $u\in H^0(X,\mathcal{O}_X^\times)$, by \Cref{prop:twisted-degrees-up-to-squares} the twisted degree $\deg^s(f,\rho) \in \sheafGW(Y \setminus f(E))$ extends to the twisted degree $\deg^u(f,\rho(E))$, which is again defined along all of $f$:
    \begin{align*}
        \deg^u(f, \rho(E)) &= f_\ast(u\cdot \rho(E)) \in \sheafGW(Y).
    \end{align*}
    As in \Cref{rmk:twisting-doesnt-commute-with-pushforward}, $\deg(f, \rho(E))$ and $\deg^u(f, \rho(E))$ are different in general.
\end{enumerate}
\end{remark}

\begin{example}\label{exa:pullback-section-is-a-square}
A notable case that combines \Cref{prop:pullback-of-twisted-deg} and \Cref{rmk:twisting-degree-by-section-of-2D} is when $D$ is not the square of a Cartier divisor, but the pullback $g^{\prime\ast}(D)$ \emph{is} a square of a Cartier divisor, say $g'^*D = 2E$. In this case $\deg^s(f, \rho)$ is the only available `degree' of $f$ and is defined only on $Y \setminus f(D)$. By \Cref{rmk:twisting-degree-by-section-of-2D}, the pullback $g^*\deg^{s}(f, \rho)$ extends to a twisted degree of $f'$, defined on all of $Y'$. We also see that $f'$ also has an `untwisted' degree $\deg(f', (g^\ast\rho)(E))$. As discussed above, the two degrees can be quite different. Indeed, this is the essential mechanism behind our main result on Wronski maps.
\end{example}

\subsection{\texorpdfstring{$\mathbb{A}^1$}{A1}-chambers} The definition of untwisted degrees motivates the following setup.  We maintain \Cref{notation:kxyf}.

Let $\rho$ be an orientation of the map $f\colon X \to Y$. If $Y$ is not $\mathbb{A}^1$-connected, the class $\deg(f,\rho) \in \sheafGW(Y)$ need not be fiber-invariant: the value of $\deg(f,\rho)(y) \in \GW(k(y))$ as in \Cref{prop:sum-of-local-degrees} may depend upon the choice of closed point $y\in Y$. In differential topology, we study the connected components of $Y$, as Brouwer degrees are constant in fibers on each connected component; these are called \emph{chambers}. In algebraic geometry, we work instead with a slightly different notion, which we will call an $\A^1$\emph{-chamber}.

\begin{definition}\label{def:A1-chamber} Let $Y$ be a $k$-variety. An $\mathbb{A}^1$-\emph{chamber} of $Y$ is a flat morphism $g\colon W \to Y$ for which $W$ is smooth and essentially $\A^1$-connected.
\end{definition}

An arbitrary variety need not admit \emph{any} $\A^1$-chambers, although in the context considered in this paper we have many natural ones to consider. The primary advantage of chambers is that points lying in the (image of the) same $\A^1$-chamber have a fiber-invariance property.

\begin{proposition}[Fiber invariance within $\mathbb{A}^1$-chambers]\label{prop:degs-same-across-a1-chamber} Let $\rho$ be an orientation of $f$. Let $g\colon W \to Y$ be an $\mathbb{A}^1$-chamber, let $L/k$ be a field extension, and let $y_1,y_2 \in Y(L)$ be two $L$-points lying in the image of $g$. Then $\deg(f,\rho)(y_1) = \deg(f,\rho)(y_2)$ in $\GW(L)$.
\end{proposition}
\begin{proof} The maps $y_1,y_2\colon \Spec(L) \to Y$ factor through $W$. By \Cref{prop:a1deg-flat-pullback}, we can compute their fiber degrees using $g^\ast \deg(f,\rho)\in \sheafGW(W)$. Since $W$ is essentially $\mathbb{A}^1$-connected, equality follows from \Cref{prop:a1connected-constancy}.
\end{proof}

Often (and in our application to Wronski maps), we cannot even define a degree in $\sheafGW(Y)$ due to $f$ being nonorientable, whereas $f$ may have a natural $D$-orientation for some Cartier divisor $D \subset X$. We therefore introduce a notion of $\A^1$-chamber relative to $D$. For this purpose, we will assume $Y$ is smooth and $X$ is generically smooth. Let $R_f \subseteq X$ be the ramification locus of $f$. Note that $R_f$ includes any singular locus of $X$ (since $Y$ is smooth).

\begin{definition}\label{def:relative-A1-chamber}
Let $E \subseteq Y$ be a divisor. An \emph{$\A^1$-chamber of $Y$ relative to $E$} is an $\A^1$-chamber $g : W \to Y$ such that $g^*E$ is the square of a divisor on $W$.

Now assume $D$ has no irreducible components contained in $R_f$. An \emph{$\A^1$-chamber of $f$ relative to $D$} is an $\A^1$-chamber of $Y$ relative to $f(D)$.
\end{definition}

We emphasize that the $\A^1$-chambers of a \emph{map} are defined relative to divisors on the domain, and only relative to divisors with no components in common with the ramification locus $R_f$.

\begin{theorem}[Fiber invariance of twisted degrees in chambers of a map]
\label{thm:fiber-invariance-twisted-degrees-on-chambers}
Let $\rho$ be a $D$-orientation of~$f$. Assume $D$ has no irreducible components contained in $R_f$, and let $g : W \to Y$ be an $\A^1$-chamber of $f$ relative to $D$. Let $w_1,w_2 \in W(k)$ and let $y_1,y_2 \in Y(k)$ be their images. Assume the $y_i$ are not in $f(D)$. Then
\begin{align*}
    \deg^s(f,\rho)(y_1) = \deg^s(f,\rho)(y_2) \text{ in } \GW(k).
\end{align*}
\end{theorem}
\begin{proof}
By hypothesis, $D \cap R_f$ has codimension $2$ in $X$. Since $f$ is finite, $Z$ has codimension $2$ in $Y$. Since $g$ is flat, $g^{-1}(Z)$ has codimension $2$ in $W$. Thus the entire setup, in particular the assumption that $W$ is essentially $\A^1$-connected, is preserved under restricting $Y$ to the open set $Y \setminus Z$. Performing this base change, we reduce to the case where $D \cap R_f$ is empty.

We now consider the pullback diagram
\[\begin{tikzcd}
    X_W \rar["g_X"]\dar["f_W" left]\pb & X\dar["f"] \\
    W\rar["g" below] & Y.
\end{tikzcd}\]
We claim that $g_X^*D$ is the square of a Cartier divisor. To see this, we first note that the singular locus of $X_W$ is 
\[R_g \times_Y R_f = f_W^{-1}(R_g) \cap g_X^{-1}(R_f).\]
Since $D \cap R_f$ is empty, we see that $g_X^{-1}(D)$ does not meet this locus, i.e. is contained in the smooth locus of $X_W$. Thus, it suffices to check that $g_X^*D$ is the square of a Weil divisor. For this: since $g$ is a chamber of $f$ relative to $f(D)$, $g^*(f(D)) = 2 \widetilde{D}$ for some divisor $\widetilde{D}$ on $W$. Pulling back to $X_W$, we get
\[g_X^* f^* (f(D)) = f_W^* g^* (f(D)) = 2 f_W^* \widetilde{D}.\]
On the other hand, since $D$ has no components contained in $R_f$, we may write
\[f^*(f(D)) = D + E\]
for some divisor $E \subseteq X$ with no components in common with $D$. Pulling back along $g_X$ yields
\[g_X^*f^*(f(D)) = g_X^*D + g_X^*E.\]
Combining, we see that every component of $g_X^*D$ has even multiplicity, as claimed, so
\[g_X^*D = 2H\]
for some Cartier divisor $H$ on $X_W$ (we are using that $X_W$ is smooth, in particular locally factorial). Let $t$ be an equation for $H$ and write $g_X^* s = ut^2$ for some global invertible function $u\in H^0(X_W,\mathcal{O}_{X_W}^\times)$. By \Cref{prop:pullback-of-twisted-deg},
\begin{align*}
    g^\ast \deg^s(f,\rho) = \deg^{ut^2}(f_W, g_X^\ast \rho) \text{ in } \sheafGW(W\setminus g^{-1}(f(D))).
\end{align*}
By \Cref{prop:twisted-degrees-up-to-squares}, this extends to the class
\begin{align*}
    \deg^u(f_W, (g_X^\ast\rho)(-2H)),
\end{align*}
defined on all of $W$. Since $W$ is essentially $\A^1$-connected, this class is independent of the choice of $k$-rational point we pull it back to. This class is therefore same on $w_1$ and $w_2$, and since $y_1$ and $y_2$ lie off of $f(D)$, it can be identified with the twisted degree evaluated at $y_1$ or $y_2$.
\end{proof}
\begin{remark}
More generally, we see that $\deg^s(f, \rho) \in \sheafGW(W)$ comes from a unique class $\overline{\beta} \in \GW(k)$. Thus if $w \in W$ is an $L$-point for some finite extension $L/k$, and $y = g(w)$, we obtain $\deg^s(f, \rho)(y) = p^*\overline{\beta}$. We note that $\GW(L)$ in general has less information than $\GW(k)$, so the theorem is most useful when applied to $k$-points.
\end{remark}
\begin{remark}[Twisted analogue of \Cref{ex:proper-geometrically-connected-deg-independent-of-orientation}]\label{rmk:proper-geometrically-connected-twisted-deg-independent-of-everything-except-D}
If $X$ has $H^0(\mathcal{O}_X^\times) = k^\times$ (e.g. $X$ proper and geometrically connected), then any two choices of equation $s$ for $D$ differ by a scalar, so different choices of $s$ affect $\deg^s(f, \rho)$ only by multiplication by a global square class. If in addition $\Pic(X)[2] = 0$, then the choice of isomorphism $\rho$ (given the choice of $D$) likewise only affects $\deg^s(f, \rho)$ by multiplication by a global square class. Thus, when this holds, the only essential choice is the divisor $D$ itself.
\end{remark}

Thus, if $f\colon X \to Y$ is $D$-oriented and $D$ has no components in the ramification locus of $f$, the twisted degree $\deg^s(f, \rho)$ is invariant along points in the same $\A^1$-chamber of $f$ \emph{relative to $D$}. We will leverage this observation to define twisted degrees of the Wronski map for many natural choices of chamber, in particular, enough chambers to cover all the $k$-points of $Y$.

We finish this section by giving a worked example based on the (simplified) structure of the Wronski map. This example is a toy version of \cite[Fig.~10]{LevinsonPurbhoo}.

\begin{figure}[b]
\newcommand{\numsamples}{30}
\newcommand{\thickness}{ultra thick}
\newcommand{\inverseparab}{plot[variable=\y,domain=-1.58:1.58,samples=\numsamples] ({\y*\y-1},{\y/2})}
\newcommand{\parab}{plot[variable=\y,domain=-1.58:1.58,samples=\numsamples] ({\y*\y-1},{-\y/2})}

\newcommand{\inverselowerhyperb}{plot[variable=\t,domain=-3.14:-0.31,samples=\numsamples] ({(\t - 1/\t)/2},{(\t + 1/\t)/4.5})}
\newcommand{\inverseupperhyperb}{plot[variable=\t,domain=0.31:3.14,samples=\numsamples] ({(\t - 1/\t)/2},{(\t + 1/\t)/4.5})}
\newcommand{\lowerhyperb}{plot[variable=\t,domain=-3.14:-0.31,samples=\numsamples] ({-(\t - 1/\t)/2},{(\t + 1/\t)/4.5})}
\newcommand{\upperhyperb}{plot[variable=\t,domain=0.31:3.14,samples=\numsamples] ({-(\t - 1/\t)/2},{(\t + 1/\t)/4.5})}

\newcommand{\rightcircleCW}{plot[variable=\t,domain=-1:1,samples=\numsamples] ({(1 - \t*\t)/(1 + \t*\t)},{2*\t/(1 + \t*\t)/2})}
\newcommand{\leftcircleCCW}{plot[variable=\t,domain=-1:1,samples=\numsamples] ({-(1 - \t*\t)/(1 + \t*\t)},{2*\t/(1 + \t*\t)/2})}


\begin{tabularx}{\textwidth}{YYYY}
\begin{tikzpicture}
\begin{scope}
  \clip (-1.5,-2) rectangle (1.5,2);
  \begin{scope} \clip (0,-2) rectangle (1.5,2); \draw[\thickness,blue,orient=0.8cm] \inverseupperhyperb; \end{scope}
  \begin{scope} \clip (0,-2) rectangle (1.5,2); \draw[blue,orient=0.8cm] \inverselowerhyperb; \end{scope}
  \begin{scope} \clip (-1.5,-2) rectangle (0,2); \draw[\thickness,blue,orient=0.8cm] \upperhyperb; \end{scope}
  \begin{scope} \clip (-1.5,-2) rectangle (0,2); \draw[blue,orient=0.8cm] \lowerhyperb; \end{scope}
\end{scope}
\filldraw[fill=white] (0, 0.44) circle (2pt) node [below] {\tiny $\langle +1 \rangle$}; 
\filldraw[fill=white] (0,-0.44) circle (2pt) circle (3pt) node [above] {\tiny $\langle +1 \rangle$};   

\draw[-{Stealth[length=3mm]},thick] (0,-1.25) -- (0,-2) node[pos=0.5,right=3pt] {$f'$}; 
\draw[thick,orient=1.75cm] (0,-2.3) -- (1.5,-2.3); 
\draw[thick,orient=1.75cm] (0,-2.3) -- (-1.5,-2.3); 
\filldraw[fill=white] (0,-2.3) circle (2pt);   

\end{tikzpicture}
&
\begin{tikzpicture}
\begin{scope}
  \clip (-1.5,-2) rectangle (1.5,2);
  \begin{scope} \clip (0,-2) rectangle (1.5,2); \draw[\thickness,red,orient=1.35cm] \rightcircleCW; \end{scope}
  \begin{scope} \clip (-1.5,-2) rectangle (0,2); \draw[\thickness,red,orient=1.35cm] \leftcircleCCW; \end{scope}
\end{scope}
\filldraw[fill=white] (0, 0.5) circle (2pt) node [below] {\tiny $\langle +1 \rangle$}; 
\filldraw[fill=white] (0,-0.5) circle (2pt) circle (3pt) node [above] {\tiny $\langle -1 \rangle$};   

\draw[-{Stealth[length=3mm]},thick] (0,-1.25) -- (0,-2) node[pos=0.5,right=3pt] {$f'$}; 
\draw[thick,orient=1.75cm] (1.5,-2.3) -- (0,-2.3); 
\draw[thick,orient=1.75cm] (-1.5,-2.3) -- (0,-2.3); 
\filldraw[fill=white] (0,-2.3) circle (2pt);   

\end{tikzpicture}
&
\begin{tikzpicture}
\begin{scope}
  \clip (-1.5,-2) rectangle (1.5,2);
  \begin{scope} \clip (0,0) rectangle (1.5,2); \draw[\thickness,blue,orient=0.94cm] \inverseparab; \end{scope}
  \begin{scope} \clip (-1.5,-2) rectangle (0,2); \draw[\thickness,red,orient=0.94cm] \inverseparab; \end{scope}
  \begin{scope} \clip (0,-2) rectangle (1.5,0); \draw[blue,orient=0.94cm] \parab; \end{scope}
  \draw[dashed] (0,-1) -- (0, 1);          
\end{scope}
\filldraw[fill=white] (0, 0.5) circle (2pt); 
\filldraw[fill=white] (0,-0.5) circle (2pt) circle (3pt);   

\draw[-{Stealth[length=3mm]},thick] (0,-1.25) -- (0,-2) node[pos=0.5,right=3pt] {$f$}; 
\draw[thick,orient=1.6cm] (-1.5,-2.3) -- (1.5,-2.3); 
\filldraw[fill=white] (0,-2.3) circle (2pt);   

\end{tikzpicture}
&
\begin{tikzpicture}
\begin{scope}
  \clip (-1.5,-2) rectangle (1.5,2);
  \draw[blue,orient=0.94cm] \parab;
\end{scope}

\draw[-{Stealth[length=3mm]},thick] (0,-1.25) -- (0,-2) node[pos=0.5,right=3pt] {$f$}; 
\draw[thick,orient=1.6cm] (-1.5,-2.3) -- (1.5,-2.3); 

\fill (0, 0.5) circle (2pt) node [below] {\tiny $\langle -1 \rangle$};
\fill (0, -0.5) circle (2pt) node [above] {\tiny $\langle +1 \rangle$};

\end{tikzpicture} \\
    (a) Chamber $y = z^2$ & (b) Chamber $y = -z^2$ & (c) Twisted orientation & (d) Untwisted orientation
\end{tabularx}

\caption{The toy \Cref{ex:toy}. {\bf Figures (d) and (c):} The original orientation and its twist by $x-1$, which reverses the orientation for the points $x < 1$ (drawn in thick) and requires us to delete $x=1$ (doubled circle) and the fiber containing it (dashed line). The curve divides into two chambers, shown in red and blue. {\bf Figures (a) and (b):} The pullbacks of (c) along $y = z^2$ and $y = -z^2$, showing doubled copies of the right and left halves of (c). The map is oriented over each chamber, and its orientation even extends back over the deleted fiber. The local degrees at $x = \pm1$ are shown. }
\label{fig:chambers}
\end{figure}
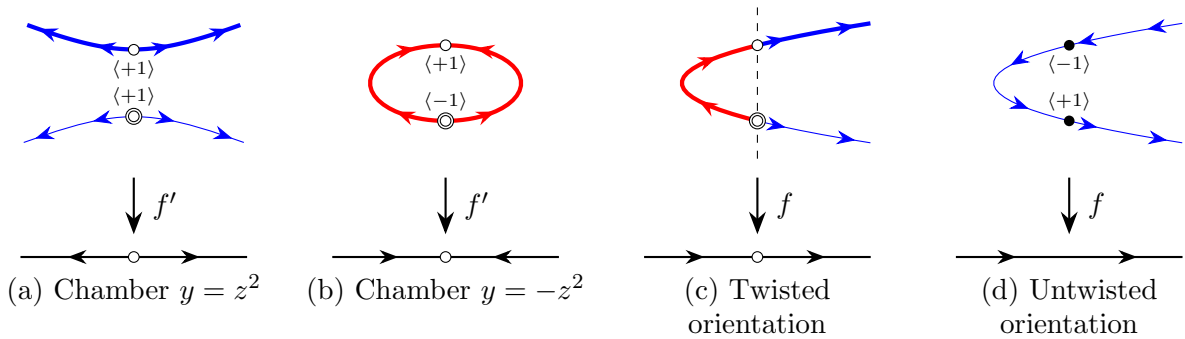

\begin{example}[A twisted degree with two chambers] \label{ex:toy}
Let $X = Y = \A^1$ with coordinates $x$ and $y$, and let $f : X \to Y$ be the map $x \mapsto y = x^2 - 1$. For simplicity we may take $k = \R$.

The map is orientable since $\omega_{\A^1}$ is trivial, and the orientation $\rho$ is unique up to a global scalar since $\Pic(X)[2] = 0$ and $H^0(\mathcal{O}_X^\times) \cong k^\times$ (cf. \Cref{ex:proper-geometrically-connected-deg-independent-of-orientation}). Since $Y$ is $\A^1$-connected, $\deg(f, \rho)$ can be found by checking a single fiber; note that the Jacobian is just $\frac{dy}{dx} = 2x$. Taking $y=0$, the local degrees at $x=\pm 1$ are $\langle 2\rangle$ and $\langle -2 \rangle$, so $\deg(f, \rho) = \langle 2, -2 \rangle = \langle 1, -1 \rangle$, the hyperbolic form. 

We now twist the orientation by exactly one of the two points above $y=0$ (see \Cref{fig:chambers}). We let $D = \{x=1\} \subset \A^1$ and consider $\rho$ as a $D$-orientation (since all line bundles are trivial on $\A^1$); we twist by the equation \[s = x-1.\] We note that $s$ is not a square, and $s$ is again unique up to global scalar. We also note that $R_f = \{x=0\}$, so $D \cap R_f = \varnothing$.

Now $\deg^s(f, \rho)$ is defined on $\A^1 \setminus f(D) \cong \mathbb{G}_m$, which is $\mathbb{A}^1$-disconnected, so a priori we might not expect any fiber invariance. In fact, the fiber degrees $\deg^s(f, \rho')(y)$ for $y \in Y(k)$ take on only two possible values, coming from two $\A^1$-chambers of $f$ relative to $D$.

We first take $g : Y' = \A^1 \to \A^1$ by $z \mapsto y = z^2$. Let
\[X' = \A^1 \times_{g,f} \A^1 \cong \frac{k[z, x]}{z^2 = x^2-1} \cong k[x+z, (x+z)^{-1}] \cong \mathbb{G}_m 
\text{ \quad and \quad }
\begin{tikzcd}
    X'\rar["g'"]\dar["f'" left] & \A^1\dar["f" right] \\
    \A^1\rar["g" below] & \A^1.
\end{tikzcd}\]
This is an $\A^1$-chamber of $f$ relative to $D$ since $D$ lies over the ramification point of $g$. Thus $g'^*D$ is a square. Indeed, $x-1$ cuts out a single point on $x' \in X'$, of multiplicity $2$. By \Cref{rmk:twisting-degree-by-section-of-2D}, $\deg^s(f, \rho)$ pulls back and extends to a twisted degree defined on all of $Y'$ (which is $\A^1$-connected!) of the form $\deg^u(f', g'^*\rho)$, where $g'^*s = u t^2$ for some $t$ cutting out $x'$ with multiplicity $1$ and $u$ a unit on $X'$.

We compute the twisted local degree using the fiber $z=0$. A straightforward calculation shows that the ideal of $x'$ has generator $z+x-1$, but since we are free to use different choices of equation $t$ for $x'$ on different charts (cf. \Cref{rmk:twisting-by-a-section-of-2D}), it is more convenient to take $t=z$ near $x=1$ and $t=1$ near $x=-1$. The defining equation $y = (x-1)(x+1)$ yields $z^2 = (x-1)(x+1)$ on $X'$, or simply
\[x-1 = \frac{1}{x+1} z^2.\]
Thus the local degree at $x=1$ is twisted by the local unit $u = s/t^2 = \tfrac{1}{x+1} = \tfrac{1}{2}$, and the local degree at $x=-1$ is twisted by $u = s/t^2 = x-1 = -2$. We find
\[\deg^{g'^*s}(f, g^* \rho')(0) = \langle \tfrac{1}{2} \cdot 2, (-2) \cdot (-2) \rangle = \langle 1, 4 \rangle = \langle 1, 1 \rangle,\]
which is not the hyperbolic form. Indeed, over $\R$, this forces the fiber over every $\R$-point of the chamber to consist always of two real points (see \Cref{fig:chambers}(a)).

In contrast, on the chamber $h$ given by $z \mapsto y = -z^2$, we get
\[h'^*s = \frac{-1}{x+1} z^2.\]
In this case, the twisted local degree at $x = 1$ is now multiplied by $\tfrac{-1}{x+1} = -\tfrac{1}{2}$. At $x=-1$ the twist is unchanged, using $u=s=x-1$ again. We see that
\[\deg^{g'^*s}(f, g^* \rho')(0) = \langle -1, 1 \rangle,\]
agreeing with the untwisted degree (but with the two points contributing opposite values). In this chamber there are fibers containing two real points of opposite signs (for $-1 < y < 0$) and complex fibers (when $y < -1$). See \Cref{fig:chambers}(b).

These two chambers account for every real fiber of $f$, since $y \in Y(\R)$ comes from the first chamber if $y > 0$ and from the second chamber if $y < 0$. By fiber invariance within $\A^1$-chambers, every real fiber of $f$ has one of the two given twisted degrees. This calculation is summarized in the table below.

\begin{center}
\begin{tabular}{c|c|c} & local degree at $x=1$ & local degree at $x=-1$ \\ \hline
untwisted, on $X = \A^1$ & $\langle 2 \rangle$ & $\langle -2 \rangle$ \\ \hline
$s$-twisted, $y = z^2$ & $\langle 1 \rangle$ & $\langle 1 \rangle$ \\ \hline
$s$-twisted, $y = -z^2$ & $\langle -1 \rangle$ & $\langle 1 \rangle$ \\
\end{tabular}
\end{center}
Over an arbitrary field, an analogous calculation on the chamber $z \mapsto y = \lambda z^2$ yields the degree $\langle \lambda, 1 \rangle$.

Topologically, we can understand the degree change as coming from simply twisting the orientation of the source curve by the function $s = x-1$. The two chambers correspond to the positive and negative parts of the real line in $y$, and we see that the Brouwer degree is $0$ in the negative chamber and $2$ in the positive chamber. The real picture is depicted in \Cref{fig:chambers}.
\end{example}

%% file: sections/osculating_schubert_calc.tex
\section{Osculating Schubert calculus}
We recap the basic definitions of Schubert varieties and osculating Schubert calculus, and we formally define the Wronski map. We construct the domino varieties as in \cite{LevinsonPurbhoo}, and the main result in this section is a proof that the Wronski map is $\Zdomino$-oriented (\Cref{lem:wronski-relatively-oriented-relative-to-vertdomino}).

\subsection{Schubert varieties and Grassmannians}

Fix a vector space $W$ over $k$ and an integer $d \leq \dim W = d+m$. We denote by $\Gr_k(d, W)$ the Grassmannian of $d$-dimensional subspaces of $W$, or $\Gr_k(d, d+m)$ if $W$ is understood. Let $F_\bullet = F_0 \subset F_1 \subset F_2 \subset \cdots \subset F_{d+m}$ be a complete flag in $W$, with $\dim F_i = i$. 

Let $\lambda = (\lambda_1 \geq \cdots \geq \lambda_d)$ be a Young diagram contained in a $d \times m$ rectangle, where $\lambda_i$ is the length of the $i$-th row. We write $|\lambda| = \sum_{i=1}^d \lambda_i$ for the size of $\lambda$ and $\lambda^c = (m-\lambda_d, \ldots, m-\lambda_1)$ for the complement of $\lambda$.

The {\bf Schubert cell of type $\lambda$ with respect to $F_\bullet$} is denoted $X_\lambda^\circ F_\bullet$ and consists of those $V \in \Gr_k(d, d+m)$ such that, for $j = 1, \ldots, d+m$,
\[
\dim (V \cap F_j) - \dim (V\cap F_{j-1}) = \begin{cases}
1 & j=m+i-\lambda_i \text{ for some }i \\
0 & \text{else.}
\end{cases}
\]
Its closure is the {\bf Schubert variety of type $\lambda$ with respect to $F_\bullet$} and is denoted $X_\lambda F_\bullet$; it is a subvariety of codimension $|\lambda|$. The following cases will be of particular importance:
\begin{enumerate}
\item $\lambda = \medsquare$ : We have $V \in X_\square F_\bullet$ if and only if $V \cap F_m \ne 0$;
\item $\lambda = \medvertdomino$ : We have $V \in X_\vertdomino F_\bullet$ if and only if $\dim(V \cap F_{m+1}) \geq 2$;
\item $\lambda = \medhorizdomino$ : We have $V \in X_\horizdomino F_\bullet$ if and only if $V \cap F_{m-1} \ne 0$;
\item $\lambda = \medsquare^c$ : We have $V \in X_{\square^c} F_\bullet$ if and only if $F_{d-1} \subset V \subset F_{d+1}$.
\end{enumerate}
We write $[X_\lambda]$ for the cycle class of $X_\lambda F_\bullet$, which is independent of the choice of $F_\bullet$. We recall that $\mathrm{Pic}(\Gr(d, d+m)) \cong \mathbb{Z}$, generated by the class of the Schubert divisor $[X_\square]$.

\subsection{The Wronski map}

\begin{notation}
For a field $k$, we denote by $k_r[z]$ the $(r+1)$-dimensional vector space of univariate polynomials of degree at most $r$. We denote by $\poly{r}{k}$ the affine space of univariate \emph{monic} polynomials of degree $r$ over $k$. As an algebraic variety, $\poly{r}{k} = \A^r_k$.
\end{notation}

\begin{definition}\label{def:wronski} Let $k$ be a field, and let $d,m\ge1$ be arbitrary integers. Then we define the \emph{classical Wronski polynomial} of $d$ univariate polynomials $f_1,\ldots,f_d\in k_{d+m-1}[z]$ to be the polynomial given by the determinant
\begin{align*}
    \Wr(f_1,\ldots,f_d)(z) = \det \begin{vmatrix}
        f_1(z) & f_2(z) & \cdots & f_d(z) \\
        f_1'(z) & f_2'(z) & \cdots & f_d'(z) \\
        \vdots & \vdots & \ddots & \vdots \\
        f_1^{(d-1)}(z) & f_2^{(d-1)}(z) & \cdots & f_d^{(d-1)}(z)
    \end{vmatrix}
\end{align*}
When $f_1,\ldots,f_d$ are linearly independent, we can think about them as determining a point in $\Gr_k(d,d+m)$, and the Wronski polynomial as living in $\mathbb{P}_{dm}[z]$. Thus in characteristic zero, we obtain a \emph{Wronski map} which is a morphism of schemes
\[
\Wr \colon \Gr_k(d,d+m) \to \P^{dm}.
\]
\end{definition}
In positive characteristic, however, by successively taking derivatives we run the risk of coefficients vanishing. It is more desirable in this setting to replace the derivative by the Hasse--Schmidt derivative \cite{Hasse1937} which is an alternative construction, defined as the $k$-linear extension of the map
\begin{align*}
    D^i \colon k_r[z] &\to k_r[z] \\
    t^j &\mapsto \binom{j}{i} t^{j-i}.
\end{align*}
If $\mathrm{char}(k) > i$ or is zero, then $D^i = \frac{1}{i!} \big( \frac{\partial}{\partial t} \big)^i$.
\begin{definition}[{\cite[Equation~(6)]{Laksov}}]
\label{def:Hasse-Wronski}
Let $k$ be a field and $d,m\ge0$ integers. We define the \emph{(Hasse) Wronski polynomial} to be
\[
\Wr(f_1,\ldots,f_d)(t) = \det \left( D^{i-1} f_j\right)_{1\le i,j\le d}.
\]
\end{definition}
Thus if $\mathrm{char}(k) \geq d$, this is just the previous definition divided by $\tfrac{1}{1!2!\cdots (d-1)!}$. We have the following:

\begin{proposition}[{\cite[Theorem~15(i)]{Laksov}}]
\label{prop:Wronski-well-defined}
Let $d,m\ge0$ and let $k$ be a field of characteristic zero or $p\ge d+m$. Then the \emph{(Hasse) Wronski map} gives a well-defined morphism of $k$-varieties
\[
\Wr \colon \Gr_k(d,d+m) \to \P_k^{dm}.
\]
\end{proposition}
\begin{proof} Let $N \subseteq \Gr_k(d,d+m)$ be the locus where the Wronski of the input polynomials vanishes:
\[
    N = \{V\in \Gr_k(d,d+m) \mid \Wr(V) = 0\}.
\]
We show that $N$ is empty. We can identify $k_{d+m-1}[t] \cong H^0(\P^1_k, \mathcal{O}(d+m-1))$, at which point it is clear after homogenizing that $\GL_2$ acts on binary forms via $\Sym^{d+m-1}$. From this perspective, we see that $\Gr_k(d,d+m)$ inherits a $\PGL_2$-action, and $N$ is preserved by this action. 

Thus, if $N$ is nonempty, by the Borel fixed point theorem it admits a Borel fixed point, i.e. one that is fixed by the torus action $t \mapsto (1+\epsilon) t$ and the translation $t \mapsto t+\epsilon$ (where we may take $\epsilon^2 = 0$). On monomials, these actions send $t^i \mapsto (1+i \, \epsilon) t^i$ and $t^i \mapsto t^i + i \, \epsilon\, t^{i-1}$. Our characteristic assumptions now imply that if $V$ is fixed by the torus, then $V$ is spanned by monomials, and if $V$ is moreover fixed by translation, then $t^k \in V$ implies $t^{k-1} \in V$ for any $k$. Thus $V = \text{span}\{t^{d-1}, \ldots, 1\}$ is the unique fixed point. But its Wronski matrix is then upper-triangular with determinant $1 \ne 0$. Hence $N$ is empty.
\end{proof}
In cases where the classical Wronski map is well-defined, it agrees with the Hasse Wronski map, hence we refer to this as the \emph{Wronski map} without loss of generality. We make the standing assumption throughout that our base field $k$ is characteristic zero or characteristic $\ge d+m$.

\begin{proposition} Let $k$ be a field of characteristic $0$ or $\ge d+m$. Then the Wronski map is flat, finite and lci, and of degree $\#\SYT(d\times m)$, where $\SYT(d\times m)$ is the number of standard Young tableaux of rectangular shape $d\times m$.
\end{proposition}
\begin{proof} Since $\Wr$ is proper and $\Wr^\ast \mathcal{O}(1) = [X_\square]$ is ample, we can conclude that $\Wr$ is finite. It is lci by \cite[0E9K]{Stacks}, and it is flat by miracle flatness \cite[00R4]{Stacks}. The degree is therefore the self-intersection of the class $[X_\square]$, which is well-known to be $\#\SYT(d \times m)$.
\end{proof}

Following \cite{Laksov}, we make the following definition:

\begin{definition}\label{def:osc-flag} Let $k$ be a field of characteristic zero or characteristic $\ge d+m$. For any $t_0\in \A^1$ we define the \emph{osculating flag}
\[
    \Osc_i(t_0) := \{f\in k_{d+m-1}[t] \mid (D^{j}f)(t_0) = 0 \text{ for all }j\le d+m-i\}.
\]
Under our characteristic assumptions, this is, equivalently, the subspace of polynomials divisible by $(t-t_0)^{d+m-i}$.
For $t_0 = \infty \in \P^1$, we define
\[
    \Osc_i(\infty) = k_{i-1}[t].
\]
\end{definition}

As $t \in \mathbb{P}^1$ varies, this definition makes $\Osc_\bullet$ a flag of subbundles of the trivial vector bundle $k_{d+m-1}[t] \otimes \mathcal{O}_{\mathbb{P}^1} \cong \mathcal{O}_{\mathbb{P}^1}^{\oplus d+m}$, in particular $\Osc_i(t_0)$ is the image of the map
\[k_{i-1}[t] \otimes \mathcal{O}_{\mathbb{P}^1}(-(d+m-i)) \xrightarrow{\cdot(t-t_0)^{d+m-i}} k_{d+m-1}[t] \otimes \mathcal{O}_{\mathbb{P}^1}.\]
We note therefore that 
\begin{equation} \label{eqn:osculating-flag-as-bundle}
\Osc_i \cong \mathcal{O}_{\mathbb{P}^1}(-(d+m-i))^{\oplus i}.
\end{equation}

Osculating Schubert calculus is in general interested in solutions to Schubert problems defined with respect to osculating flags.

\subsection{The discriminant, domino varieties, and the branch locus} Generically, a point in the target space of the Wronski map gives a Wronski polynomial with no repeated roots. We describe the closed locus of $\P^n$ of points that do not have this property:

\begin{notation} Let $\Delta_n \subseteq \mathbb{P}^n$ denote the \emph{discriminant locus}, defined as those univariate polynomials of degree $n$ that have a repeated root. We note that $\Delta_n$ is a hypersurface of degree $2n-2$.
\end{notation}

Fixing $t_0 \in \P^1$, we can likewise consider the subspace of polynomials with a root at $t_0$ of some given order. The preimage of this locus, as a cycle, admits the following combinatorial description in terms of Schubert varieties with respect to the osculating flag at $t_0$.

\begin{proposition}\label{prop:wronski-root-schubert}
Let $k$ be of characteristic zero or $\ge d+m$. Let $V \in \Gr_k(d, d+m)$ and $t_0 \in \P^1$. Then $\Wr(V)$ has a root of order $k$ at $t_0$ if and only if $V \in X_\lambda \Osc_\bullet(t_0)$ for some $\lambda$ of size $k$.

Moreover, letting $W'(t_0) \subset \P^{dm}$ be the subspace of polynomials divisible by $(t-t_0)^k$, we have the equality of cycles
\[[\Wr^{-1}(W'(t_0))] = \sum_{|\lambda| = k} |\mathrm{SYT}(\lambda)| \cdot [X_\lambda \Osc_\bullet(t_0)].\]
\end{proposition}
\begin{proof} This was stated over the complex numbers in \cite[Theorem 2.5 and Corollary 2.6]{purbhooJeuTaquinMonodromy2010}, and is essentially the content of \cite[Theorem~15]{Laksov}. We provide a self-contained proof here.

Let $V\in \Gr_k(d,d+m)$ and suppose $V \in X_\lambda^\circ \Osc_\bullet(t_0)$. Let $a_j = j-1 + \lambda_{d+1-j}$. (The sequence $a_1 < \cdots < a_d$ is the \emph{vanishing sequence} of $V$ at $t_0$). Then $V$ has a basis $\{f_1, \ldots, f_d\}$ of the form
\begin{align*}
    f_j &= \sum_{k=a_j}^{d+m-1} b_{jk} (t-t_0)^k, \quad \text{with } b_{j,a_j} = 1, \\
    &= (t-t_0)^{a_j} + \text{(higher terms)}.
\end{align*}
We show that the order of vanishing of $\Wr(V)$ at $t_0$ is $|\lambda|$. The matrix $(D^{i-1}f_j)$ factors as $B\cdot M$, where $B = (b_{jk})$ is $d\times (d+m)$ and $M = (D^{i-1}(t-t_0)^{k})$ is $(d+m)\times d$. Via the Cauchy--Binet formula we can expand its determinant as
\begin{align*}
    \Wr(V) &= \sum_{K = \{k_1 < \cdots < k_d\}} \det(B_K) (t-t_0)^{(\sum_{i=1}^d k_i) - \binom{d}{2}} \frac{\prod_{i<i'}(k_i - k_{i'})}{\prod_{i=0}^{d-1}i!}.
\end{align*}
Since $1\le k_i \le d+m-1$, we have that $k_i - k_{i'}$ never vanishes in characteristic $p$. The order of vanishing is then governed by when $\sum_{i=1}^d k_i$ is minimized, subject to the constraint that $\det(B_K)\ne 0$. This occurs at $(k_1,\ldots,k_d) = (a_1,\ldots, a_d)$, with $\det(B_K) = 1$ by upper-triangularity.

Therefore set-theoretically we have equality
\begin{align*}
    \Wr^{-1}(W'(t_0)) = \bigcup_{|\lambda| = k} X_\lambda \Osc_\bullet(t_0).
\end{align*}
Define $H(t_i) \subseteq k_{md}[t]$ to be the subspace of polynomials divisible by $t-t_i$. Then $W'(t_0)$ is rationally equivalent to $W'' = H(t_1) \cap \cdots \cap H(t_k)$, where $t_1, \ldots, t_k$ are distinct. By flatness, we have the equality of cycle classes
\[
    [\Wr^{-1}(W'(t_0))] = [\Wr^{-1}(W'')] = \prod [\Wr^{-1}(H(t_i))] = \Wr^*[\mathcal{O}(1)]\medsquare^k = \sum_{|\lambda| = k} |\mathrm{SYT}(\lambda)| [X_\lambda].
\]
Since the Schubert classes form a basis for the Chow group, we conclude that the multiplicity of $X_\lambda \Osc_\bullet(t_0)$ in the fundamental cycle of $\Wr^{-1}(W'(t_0))$ is $|\SYT(\lambda)|$.
\end{proof}

The preimage of the entire Wronski map can also be described, using what are called \emph{domino varieties}, following \cite[\S3.2]{LevinsonPurbhoo}.

\begin{definition} Given a Young diagram $\lambda \subseteq d\times m$, we define $Z_\lambda$ to be the union of the Schubert varieties of type $\lambda$ over all times $t\in \P^1$:
\[
Z_\lambda := \bigcup_{t\in \P^1_k} X_\lambda \Osc_\bullet(t) \subseteq \Gr(d, d+m).
\]
This is an irreducible closed subvariety of codimension $|\lambda| - 1$.
\end{definition}

\begin{terminology} For $\lambda = \medvertdomino$ or $\lambda = \medhorizdomino$, we call the associated varieties $\Zdomino$ and $Z_\horizdomino$ the \emph{domino varieties}. These are important as they describe the preimage of the discriminant under the Wronski map.
\end{terminology}

\begin{proposition}\label{prop:preimage-of-discriminant-locus} The scheme-theoretic preimage of $\Delta_{dm}$ is the reduced union
\[
\Wr^{-1}(\Delta_{dm}) = \Zdomino \cup Z_\horizdomino,
\]
and both $\Zdomino$ and $Z_\horizdomino$ map surjectively onto $\Delta_{dm}$ by the Wronski map.
\end{proposition}
\begin{proof}
Let $V \in \Gr_k(d,d+m)$. By Proposition \ref{prop:wronski-root-schubert}, $\Wr(V)$ has a double root at $t \in \P^1$ if and only if $V \in X_\lambda \Osc_\bullet(t)$ for some $\lambda$ with $|\lambda| = 2$. This gives the statement above as an equality of sets. Letting $t$ be arbitrary and restricting further to the space $W'(t)$ of polynomials divisible by $(z-t)^2$, we see that $[\Wr^{-1}(W'(t))]$ is the cycle $[X_{\vertdomino}\Osc_\bullet(t)] + [X_{\horizdomino} \Osc_\bullet(t)]$ and is generically reduced. It follows that $\Wr^{-1}(\Delta_{dm})$ is generically reduced, hence reduced (since a Cartier divisor on a smooth variety is Cohen--Macaulay). Thus equality holds scheme-theoretically.
\end{proof}

We will prove that the Wronski map is $\Zdomino$-oriented. In order to show this, we first need to know the divisor class of $\Zdomino$:

\begin{proposition}\label{prop:degree-domino-variety}
The class of $\Zdomino \subseteq \Gr_k(d,d+m)$ is $(d-1)(m+1) \cdot [X_\square]$.
\end{proposition}

This statement will follow by taking $\lambda = \medvertdomino$ in the following more general calculation:

\begin{proposition} \label{prop:kempf-laksov-degree-calculation}
Let $\mathcal{X}_\lambda \subset \Gr(d, d+m) \times \mathbb{P}^1$ be the incidence correspondence
\[\mathcal{X}_\lambda = \{(V, t) : V \in X_\lambda \Osc_\bullet(t)\},\]
i.e. the universal osculating flag Schubert variety of type $\lambda$. Then 
\begin{equation} \label{eqn:universal-osculating-schubert-variety-class}
[\mathcal{X}_\lambda] = \pi_1^*[X_\lambda] + \sum_{\lambda' \subset \lambda} (d-i+\lambda_i)(m+i-\lambda_i) \cdot \pi_1^*[X_{\lambda'}] \cdot \pi_2^*[\mathrm{pt}],
\end{equation}
where the sum runs over all $\lambda' \subset \lambda$ obtained by deleting an outer corner box of $\lambda$, and $i=i(\lambda', \lambda)$ denotes the row of the deleted box. If we furthermore assume that $|\lambda|\ge 2$, and that the characteristic of $k$ is 0 or at least $d+m$, then we obtain
\begin{equation} \label{eqn:z-lambda-class}
[Z_\lambda] = \sum_{\lambda' \subset \lambda} (d-i+\lambda_i)(m+i-\lambda_i) [X_{\lambda'}]
\end{equation}
by pushing down to $\Gr_k(d,d+m)$.
\end{proposition}

\begin{proof}
Let 
\[0 \to \mathcal{S} \to k_{d+m-1}[z] \otimes \mathcal{O}_\Gr \xrightarrow{\varphi} \mathcal{Q} \to 0\]
denote the short exact sequence of tautological bundles on $\Gr = \Gr(d, d+m)$. On $\Gr(d, d+m) \times \mathbb{P}^1$, we view $k_{d+m-1}[z] \otimes \mathcal{O}_{\Gr \times \P^1}$ as having a complete flag of subbundles given by the osculating flags,
\[\pi_2^* \Osc_1 \subset \cdots \subset \pi_2^* \Osc_{d+m} = k_{d+m-1}[z] \otimes \mathcal{O}_{\Gr \times \P^1}.\]
We consider the pullback of $\varphi$ along $\pi_1$, restricted to each of these subbundles. The rank of
\[\pi_2^* \Osc_{m+i-\lambda_i} \to \pi_1^* \mathcal{Q}\]
at a point $(V, t)$ is $m+i-\lambda_i - \dim(V \cap \Osc_{m+i-\lambda_i})$, so $\mathcal{X}_\lambda$ is the degeneracy locus where, for each $i$, this rank is at most $m - \lambda_i$. Since $\mathcal{X}_\lambda$ has the expected codimension $|\lambda|$, we may compute its class $[\mathcal{X}_\lambda]$ using the Kempf--Laksov formula. Specifically, it is given by the determinant of a $d \times d$ matrix in the Chern classes of the virtual bundles $\pi_1^* \mathcal{Q} - \pi_2^* \Osc_{m+i-\lambda_i}$, namely
\begin{equation} \label{eqn:kempf-laksov}
[\mathcal{X}_\lambda] = \det\big(c_{\lambda_i - i + j}(\pi_1^* \mathcal{Q} - \pi_2^* \Osc_{m+i-\lambda_i})\big)_{1 \leq i,j \leq d}.\end{equation}
To calculate this, we note that
\begin{align*}
c(\pi_1^* \mathcal{Q}) &= 1 + \pi_1^* [X_\square] + \pi_1^* [X_\horizdomino] + \pi_1^* [X_\smallydiagram{3}] + \cdots, \\
&= 1 + \medsquare + \medhorizdomino + \medydiagram{3} + \cdots,
    && \text{writing $\lambda$ for $\pi_1^*[X_\lambda]$,} \\
c(\pi_2^* \Osc_{m+i-\lambda_i}) &= \pi_2^* c\big(\mathcal{O}_{\P^1}(-(d-i+\lambda_i))^{\oplus m+i-\lambda_i} \big)
    &&\text{by Equation \eqref{eqn:osculating-flag-as-bundle},} \\
&= (1-(d-i+\lambda_i)\pi_2^*[\mathrm{pt}])^{m+i-\lambda_i} \\
&= 1 - \underbrace{(d-i+\lambda_i)(m+i-\lambda_i)}_{=: \, a_i} \pi_2^*[\mathrm{pt}],
    && \text{since $[\mathrm{pt}]^2 = 0$ on $\P^1$,} \\
&= 1 - a_i \epsilon,
\end{align*}
writing $\epsilon = \pi_2^*[\mathrm{pt}]$. We may then compute the virtual classes, using $\epsilon^2 = 0$ again:
\begin{align*}
c(\pi_1^* \mathcal{Q} - \pi_2^* \Osc_{m+i-\lambda_i}) &:= 
\frac{c(\pi_1^*\mathcal{Q})}{c(\pi_2^* \Osc_{m+i-\lambda_i})} \\
&= \frac{1 + \medsquare + \medhorizdomino + \medydiagram{3} + \cdots}{1 - a_i \epsilon } \\
&= \big(1 + \medsquare + \medhorizdomino + \medydiagram{3} + \cdots \big)\big(1 + a_i  \epsilon \big) \\
&= 1 + \big(\medsquare + a_i \epsilon\big) +
\big(\medhorizdomino + a_i \, \medsquare \, \epsilon\big) +
\big(\medydiagram{3} + a_i \, \medhorizdomino \, \epsilon\big) + \cdots,
\end{align*}
and we see that the entries of the Kempf--Laksov matrix \eqref{eqn:kempf-laksov} are
\begin{equation}
c_{\lambda_i - i + j}(\pi_1^* \mathcal{Q} - \pi_2^* \Osc_{m+i-\lambda_i}) = 
\underbrace{\medydiagram{7}}_{\lambda_i - i + j} \ +\ a_i \underbrace{\medydiagram{6}}_{\lambda_i - i + j - 1} \epsilon
\end{equation}
or zero if $\lambda_i - i + j < 0$.

Expanding out the determinant \eqref{eqn:kempf-laksov}, the term with no $\epsilon$ factors is just the usual Jacobi--Trudi determinant for the class $\lambda$. Since $\epsilon^2 = 0$, the only higher term is the one with a single factor of $\epsilon$. By multilinearity, the $\epsilon$ term is the sum of $d$ determinants, where the $i$-th one is obtained by always taking the $\epsilon$ term in the $i$-th row and the terms without $\epsilon$ in each other row. If $\lambda_i = \lambda_{i+1}$, the resulting matrix has a repeated row in rows $i$ and $i+1$, so its determinant is $0$. If $\lambda_i > \lambda_{i+1}$, it is instead the usual Jacobi--Trudi matrix (times the constant $a_i$) for the partition $\lambda' \subset \lambda$ obtained by decrementing $\lambda_i$. This gives the desired formula.

Finally, to obtain Equation \eqref{eqn:z-lambda-class}, we apply $(\pi_1)_*$ to Equation \eqref{eqn:universal-osculating-schubert-variety-class}, and we obtain
\begin{equation}\label{eqn:degXlambda-over-Zlambda}
    \deg(\mathcal{X}_\lambda/Z_\lambda)[\mathcal{X}_\lambda] = \sum_{\lambda' \subset \lambda} (d-i+\lambda_i)(m+i-\lambda_i) [X_{\lambda'}],
\end{equation}
where the right hand side follows by the projection formula. Thus the proposition reduces to arguing that $\mathcal{X}_\lambda \to Z_\lambda$ has degree one.

We can write a generic point $V\in Z_\lambda$ as $(z-t)^{|\lambda|}u(z)$ where $u(z)$ is a squarefree polynomial and is nonvanishing at $t$. Since $|\lambda|\ge 2$, we have that $t$ is the unique value in $\P^1$ for which $V \in X_\lambda(t)$. This tells us that the map $\mathcal{X}_\lambda \to Z_\lambda$ is generically injective on $\overline{k}$-points (since both $\mathcal{X}_\lambda$ and $Z_\lambda$ are geometrically irreducible, we can compute the degree after base change to the algebraic closure). Therefore since the size of the geometric fiber is one, this is equal to the separable degree, allowing us to conclude that $k(Z_\lambda) \subseteq k(\mathcal{X}_\lambda)$ is purely inseparable, of some degree $\deg(\mathcal{X}_\lambda/Z_\lambda) = p^r$.

By \Cref{eqn:degXlambda-over-Zlambda}, we conclude that $p^r$ divides $(d-i+\lambda_i)(m+i-\lambda_i)$ for each $\lambda' \subseteq \lambda$. Since $1\le \lambda_i \le m$ and $1\le i\le d$, the largest that either factor can be is $d+m-1$. Thus by our assumptions on $p$, we must have that $r=0$.
\end{proof}

\begin{example}
Let $\lambda = (3, 3, 1) = \medydiagram{3,3,1}$ and write $\epsilon = \pi_2^*[\mathrm{pt}]$. Let $a_i = (d-i+\lambda_i)(m+i-\lambda_i)$ be as above. The Kempf--Laksov matrix is
\begin{align*}
&\begin{bmatrix}
\medydiagram{3} + a_1 \epsilon \cdot \medydiagram{2} &
\medydiagram{4} + a_1 \epsilon \cdot \medydiagram{3} &
\medydiagram{5} + a_1 \epsilon \cdot \medydiagram{4} \\
\medydiagram{2} + a_2 \epsilon \cdot \medydiagram{2} &
\medydiagram{3} + a_2 \epsilon \cdot \medydiagram{3} &
\medydiagram{4} + a_2 \epsilon \cdot \medydiagram{4} \\
0 + a_3 \epsilon \cdot 0 &
1 + a_3 \epsilon \cdot 0 &
\medydiagram{1} + a_3 \epsilon \cdot 1
\end{bmatrix}
\\
&=
\begin{bmatrix}
\medydiagram{3} &
\medydiagram{4} &
\medydiagram{5} \\
\medydiagram{2} &
\medydiagram{3} &
\medydiagram{4} \\
0 &
1 &
\medydiagram{1}
\end{bmatrix}
+ \epsilon
\begin{bmatrix}
a_1  \cdot \medydiagram{2} &
a_1  \cdot \medydiagram{3} &
a_1  \cdot \medydiagram{4} \\
a_2  \cdot \medydiagram{2} &
a_2  \cdot \medydiagram{3} &
a_2  \cdot \medydiagram{4} \\
0 &
0 &
a_3 \cdot 1
\end{bmatrix} = A + \epsilon B
\end{align*}
The left matrix has $\det(A) = \lambda$ by the Jacobi--Trudi formula and gives the term with no $\epsilon$, while the $\epsilon$ terms are obtained by swapping one row of $B$ into $A$ each of the three possible ways, giving
\[
a_1 \det \begin{bmatrix}
\medydiagram{2} &
\medydiagram{3} &
\medydiagram{4} \\
\medydiagram{2} &
\medydiagram{3} &
\medydiagram{4} \\
0 &
1 &
\medydiagram{1}
\end{bmatrix}
+
a_2 \det
\begin{bmatrix}
\medydiagram{3} &
\medydiagram{4} &
\medydiagram{5} \\
\medydiagram{1} &
\medydiagram{2} &
\medydiagram{3} \\
0 &
1 &
\medydiagram{1}
\end{bmatrix}
+
a_3 \det
\begin{bmatrix}
\medydiagram{3} &
\medydiagram{4} &
\medydiagram{5} \\
\medydiagram{2} &
\medydiagram{3} &
\medydiagram{4} \\
0 &
0 &
1
\end{bmatrix}.
\]
The first matrix has a repeated row, so its determinant is $0$. The second and third matrices give, respectively, $\lambda' = \medydiagram{3,2,1}$ and $\lambda' = \medydiagram{3,3}$. Thus
\[
[\mathcal{X}_\smallydiagram{3,3,1}] = \medydiagram{3,3,1} + a_2 \epsilon \, \medydiagram{3,2,1} + a_3 \epsilon \, \medydiagram{3,3} \quad \text{ and } [Z_\smallydiagram{3,3,1}] = a_2\, \medydiagram{3,2,1} + a_3\, \medydiagram{3,3}.
\]
\end{example}

\begin{example} The hook and box varieties in $\Gr(d,d+m)$ have classes
\begin{align*}
    [Z_\smallydiagram{2,1}] &= (d+1)(m-1)[X_\vertdomino] + (d-1)(m+1)[X_\horizdomino] \\
    [Z_\smallydiagram{2,2}] &= dm[X_\smallydiagram{2,1}].
\end{align*}
\end{example}

Finally, we discuss orientations of the Wronski map. The key lemma of this section is:

\begin{lemma}
\label{lem:wronski-relatively-oriented-relative-to-vertdomino}
The Wronski map is not in general oriented, but is $Z_\vertdomino$-oriented.
\end{lemma}
\begin{proof} Since $\Wr$ is a central projection out of Pl\"ucker space, it pulls back $\mathcal{O}_{\P^{dm}}(1)$ to $\mathcal{O}_\Gr(1)$. Since $\omega_{\P^{dm}} \cong \mathcal{O}_{\P^{dm}}(-dm-1)$ and $\omega_\Gr \cong \mathcal{O}_\Gr(-d-m)$, we have
\[
\omega_\Wr = \Wr^* \omega_{\P^{dm}}^\vee \otimes \omega_{\Gr} = \mathcal{O}_\Gr(dm+1) \otimes \mathcal{O}_\Gr(-d-m) = \mathcal{O}_\Gr((d-1)(m-1)).
\]
In particular, $\omega_\Wr$ is a square if and only if $d$ or $m$ is odd. On the other hand, by \Cref{prop:degree-domino-variety},
\[
\omega_{\Wr} \otimes \mathcal{O}(Z_\vertdomino) \cong \mathcal{O}(2(d-1)m),
\]
so the Wronski map always admits a $Z_\vertdomino$-orientation class.
\end{proof}

As a last result we will need from this section, we check that $\Zdomino$ is not contained in the ramification locus of the Wronski map. This will allow us to set up chambers for the Wronski map relative to $\Zdomino$ as in \Cref{def:relative-A1-chamber}.

\begin{proposition}\label{prop:Zdomino-not-contained-in-ramification-locus}  Let $R \subseteq \Gr(d, d+m)$ be the ramification locus of $\Wr$. Then $\Zdomino \not\subseteq R$.
\end{proposition}
\begin{proof} By \Cref{prop:preimage-of-discriminant-locus}, $\Wr^{-1}(\Delta_{dm}) = \Zdomino \cup Z_\horizdomino$, with the reduced structure. In particular, $\Wr$ is not ramified at a general point of $\Zdomino$ (nor for that matter at a general point of $Z_\horizdomino$).
\end{proof}

%% file: sections/twisted_wr.tex
\section{Twisted degrees of Wronski maps}

We now prove our main results regarding the enriched degrees of the Wronski map. We know that the Wronski map always admits a $\Zdomino$-orientation class $\rho_\vertdomino$ by \Cref{lem:wronski-relatively-oriented-relative-to-vertdomino}, hence if $\sdomino \in \mathcal{O}(\Zdomino)$ is an equation for it, we obtain a well-defined twisted degree $\deg^{\sdomino}(f,\rho_\vertdomino) \in \sheafGW(\P^{dm}\setminus\Delta)$ as in \Cref{def:twisted-degree}. The complement of the discriminant locus in projective space is $\mathbb{A}^1$-disconnected in general, so we can obtain many different degrees by pulling back to various points. This motivates us to try to find reasonable $\mathbb{A}^1$-chambers of the Wronski map relative to $\Zdomino$, which will yield fiber-invariant twisted degrees. Our idea will be to take the ``root space'' of Wronski polynomials, and give it a twisted rational structure corresponding to the Galois orbits of roots. The orbits constructed in this way cover the complement of the discriminant locus, and our construction works over all fields of characteristic $0$ or $\geq d+m$. Thus we obtain well-defined \emph{twisted degrees} of the Wronski map, which depend only on the Galois type of the roots of the Wronskian. Over $\R$, these twists generalize the character orientations used in \cite{LevinsonPurbhoo}, and the corresponding \emph{untwisted degrees} as in \Cref{rmk:twisting-degree-by-section-of-2D} are those computed in \cite{eremenkoDegreesRealWronski2002}. We end this section by computing the twisted degrees of the Wronskian on $\Gr(2, 4)$ explicitly, over all fields of characteristic $0$ or $\geq 5$, showing that it depends precisely on the square class of the discriminant.

With all this in mind, we fix some notation for this section:

\begin{notation}\label{nota:rho-domino}
Fix $\sdomino\in \mathcal{O}(\Zdomino)$ an equation for the domino variety $\Zdomino$ on the Grassmannian and $\rhodomino$ a $\Zdomino$-orientation for the Wronski map:
\[
\rhodomino \colon \mathcal{O}(1-d)^{\otimes 2} \xto{\sim} \omega_\Wr \otimes \mathcal{O}(-\Zdomino).
\]
When $d$ or $m$ is odd, $\mathcal{O}(-\Zdomino)$ is a square, so we can use $\rhodomino$ to obtain a relative orientation for the Wronski map as in \Cref{rmk:twisting-degree-by-section-of-2D} and thereby obtain a well-defined untwisted degree. When $d$ or $m$ is odd, we call the resulting orientation of the Wronski map $\rho$.
\end{notation}

First, we construct some suitable $\mathbb{A}^1$-chambers for the Wronski map.

\subsection{Chambers from twisted root spaces} Recall that there is a natural map
\[
\Sym \colon (\P^1)^{dm} \to \P^{dm},
\]
given in affine coordinates by sending a tuple $(t_1,\ldots,t_{dm})$ to the polynomial $\prod_{i=1}^{dm}(t-t_i)$. We will refer to $(\P^1)^{dm}$ as the \emph{root space}, being the space of (ordered) roots of Wronski polynomials. We will argue that this yields an $\mathbb{A}^1$-chamber of the Wronski map relative to $\Zdomino$. We note, however, that the $k$-rational points of the chamber correspond to Wronski polynomials with all roots defined over the base field. We will want to capture degrees over Wronski polynomials whose roots are not all $k$-rational, and we will accomplish this by twisting the rational structure on the root space.

\begin{notation} For the remainder of this subsection, fix $k$ to be a field, let $k^s$ be a fixed separable closure, and take $G_s := \Gal(k^s/k)$ to be its absolute Galois group. For ease of notation, write $n=dm$ to be the dimension of the root space.
\end{notation}

Recall that for any variety $X$ defined over a field $k$, a \emph{twist} of $X$ is another $k$-variety $X'$ so that $X$ and $X'$ become isomorphic after base change to the field closure $\overline{k}$. The possible twists of $X$ are classified by the Galois cohomology group $H^1(\Gal(\overline{k}/k), \Aut(X_{\overline{k}}))$. We are interested in the particular case where $X = (\P^1_k)^n$, and where the automorphisms of $X_{\overline{k}}$ are taken to be permutations of the various projective lines in the cartesian product.

\begin{terminology} Let $\Sigma_n \le \Aut((\P^1_{\overline{k}})^n)$ be the subgroup of permutations of the copies of the projective line. For the purposes of this paper, we refer to a \emph{twisted rational structure} on $(\P^1_k)^n$ as any twist of $(\P^1_k)^n$ determined by a class in $H^1(\Gal(\overline{k}/k),\Sigma_n)$.
\end{terminology}

More explicitly, we have the following:

\begin{definition}[Twisted root space]
\label{def:sigma-action}
Let $\sigma \colon G_s \to \Sigma_n$ be any continuous group homomorphism from the absolute Galois group of $k$ to the symmetric group on $n$ elements. We define a twisted action of $G_s$ on $(\P^1_{k^s})^n$ by
\begin{align*}
    G_s \times (\P^1_{k^s})^n &\to (\P^1_{k^s})^n \\
    (g,(t_1,\ldots,t_n)) &\mapsto \sigma(g)(gt_1,\ldots,gt_n).
\end{align*}
We define $(\P^1_k)^n_\sigma$ to be the $G_s$-fixed locus of this action, which we call a \emph{twisted root space}. 

Concretely, $G_s$ acts by its Galois action and by permuting the roots according to $\sigma$. A $k$-rational point of the twisted root space (a fixed point of the action) is therefore an $n$-tuple of roots that form Galois orbits according to $\sigma$.
\end{definition}

\begin{remark} Such a $\sigma$ gives a well-defined finite \'etale $k$-algebra $A/k$, and $(\P^1_k)^n_\sigma$ can be alternatively described as the Weil restriction $\Res_{A/k}\P^1_A$ (c.f.~\cite[\S2.6]{KLSW-count}).
\end{remark}

\begin{example} If $\sigma$ is the trivial homomorphism, then $(\P^1_k)^n_\sigma = (\P^1_k)^n$ as $k$-varieties.
\end{example}

\begin{example} \label{ex:twisted-real-p1-2} When $n=2$ and $k=\mathbb{R}$, there is only a single nontrivial case to consider, namely where $\sigma : \Gal(\mathbb{C}/\mathbb{R}) \to S_2$ takes complex conjugation to the transposition swapping $1$ and $2$. A tuple $(z_1,z_2)$ is fixed by this if and only if
\[
(z_1,z_2) = (\overline{z_2},\overline{z_1}).
\]
Thus the real points of $(\P^1_{\mathbb{R}})^2_\sigma$ are the tuples of the form $(z,\overline{z})$, where $z$ is any complex number. In particular, its real locus is homeomorphic to a sphere $\mathbb{CP}^1 \cong S^2$, rather than a torus $\mathbb{RP}^1 \times \mathbb{RP}^1$.
\end{example}

\begin{example} Over a finite field $\mathbb{F}_p$, the absolute Galois group is the profinite integers. Any morphism of the form
\begin{align*}
    \sigma \colon \widehat{\mathbb{Z}} \cong \Gal(\overline{\mathbb{F}_p}/\mathbb{F}_p) \to \Sigma_n
\end{align*}
is determined by a single permutation, since $\sigma$ necessarily factors through a finite cyclic group. Moreover, if $\sigma_1$ and $\sigma_2$ are two maps $\widehat{\mathbb{Z}} \to \Sigma_n$ which send the profinite generator to two elements in the same conjugacy class of $\Sigma_n$, then it is clear that $(\mathbb{P}^1)^n_{\sigma_1}$ and $(\mathbb{P}^1)^n_{\sigma_2}$ are isomorphic as $\mathbb{F}_p$-varieties. The point of this observation is that we can encode the different possible twisted rational structures on $(\mathbb{P}_{\mathbb{F}_p}^1)^n$ concisely in terms of cycle decompositions of permutations, i.e., partitions of $n$. For instance if $n=4$, we obtain 5 different twisted rational structures, given by
\begin{align*}
    (\mathbb{P}_{\mathbb{F}_p}^1)^4= (\mathbb{P}_{\mathbb{F}_p}^1)^4_{1^4},\ (\mathbb{P}_{\mathbb{F}_p}^1)^4_{21^2},\ (\mathbb{P}_{\mathbb{F}_p}^1)^4_{2^2},\
    (\mathbb{P}_{\mathbb{F}_p}^1)^4_{31},\
    (\mathbb{P}_{\mathbb{F}_p}^1)^4_{4}.
\end{align*}
\end{example}

\begin{example} \label{ex:twisted-real-p1-n} Similarly to the pro-cyclic case over $\mathbb{F}_p$, over the reals, every twisted rational structure on $(\mathbb{P}^1_{\mathbb{R}})^n$ is, up to reordering, of the form $(\mathbb{P}^1_{\mathbb{R}})^n_{1^r 2^s}$ where $r+2s = n$. That is, the action of the absolute Galois group occurs by complex conjugation, followed by $s$ disjoint transpositions. Its real locus is $(\R\P^1)^r \times (\C\P^1)^s$.
\end{example}

\begin{example}[The coordinate ring of $(\A^1_k)^n_\sigma$]\label{ex:twisted-coordinate-ring}
We briefly describe the coordinate ring of the affine part of the twisted root space $(\A^1_k)^n_\sigma$, i.e. the open set where $t_i \ne \infty$ for all $i$. It is $\Spec$ of the invariant subring $k^s[t_1, \ldots, t_n]^{G_s}$, where $G_s$ acts by the Galois action on coefficients and by permuting $t_1, \ldots, t_n$ via $\sigma$.

If the action of $G_s$ on $\{1, \ldots, n\}$ by $\sigma$ is not transitive, $(\A^1_k)^n_\sigma$ decomposes as a product of $k$-varieties; thus we may reduce to the case where $\sigma$ is transitive. Let $H \leq G_s$ be the stabilizer of $1 \in \{1, \ldots, n\}$. Let $L$ be the fixed field of $H$; by Galois theory, $[L:k] = n$.

Let $\alpha \in L$ and let $O(\alpha)$ be its $G_s$-orbit. There is a unique $G_s$-equivariant function $c : \{1, \ldots, n\} \to O(\alpha)$ taking $1 \mapsto \alpha$. (Namely, $c(i) = g \cdot \alpha$ where $g \in G_s$ is any element for which $\sigma(g)$ takes $1$ to $i$.) Then the linear form
\[f_\alpha := \sum_{i=1}^n c(i) t_i \in L[t_1, \ldots, t_n]\]
is invariant for the action of $G_s$ by $\sigma$, since
\[g \cdot  f_\alpha = \sum_{i=1}^n g \cdot (c(i) t_i) = \sum_{i=1}^n c(\sigma(g) \cdot i) t_{\sigma(g) \cdot i} = f_\alpha.\]
Letting $\alpha_1, \ldots, \alpha_n$ be a basis for $L$ over $k$, we have
\[k^s[t_1, \ldots, t_n]^{G_s} = k^s[f_{\alpha_1}, \ldots, f_{\alpha_n}]^{G_s} = k[f_{\alpha_1}, \ldots, f_{\alpha_n}].\]
As an example, let $n=2$ and $k = \R$, with the nontrivial twisted structure of \Cref{ex:twisted-real-p1-2}. Then
\[\R[(\A^1_\R)^2_\sigma] = \R[t_1 + t_2, i(t_1 - t_2)].\]
If the roots represent a complex conjugate pair $t_i = x \pm i y$, we see that the generators correspond to $x$ and $y$ up to scalar.
\end{example}

If $X$ is a $k$-variety, and $X'$ is a twist of $X$, then $\mathbb{A}^1$-connectivity of $X$ need not imply $\mathbb{A}^1$-connectivity of $X'$. A simple example to consider is $\P^n_k$ -- the twists of this are Severi--Brauer varieties which admit a $k$-rational point if and only if they split. The non-split ones have no $k$-rational points, hence are not $\mathbb{A}^1$-connected, despite the fact that $\mathbb{P}^n_k$ is $\mathbb{A}^1$-connected. 

The twists we use in this paper are mild enough where they do not destroy connectivity. The more general statement is the following, which is a mild extension of \cite[2.43]{KLSW-count}:

\begin{proposition}\label{prop:p1sigma-a1-connected}
Let $k$ be a field, and let $A/k$ be a finite \'etale $k$-algebra of degree $n$. Then $\Res_{A/k}\P^1_A$ is $\mathbb{A}^1$-chain connected.
\end{proposition}
\begin{proof} Since $A$ is a product of finite separable field extensions $L_i/k$,
\[
    \Res_{A/k} \P^1_A \cong \prod_i \Res_{L_i/k}\P^1_{L_i}.
\]
Each of the $\Res_{L_i/k}\P^1_{L_i}$ is $\A^1$-chain connected by \cite[2.43]{KLSW-count}, and the product is $\A^1$-connected by \cite[2.44]{KLSW-count}.
\end{proof}
The same proof works to show that if $Z$ is smooth and proper over $A$ and $\mathbb{A}^1$-chain connected after base change to any factor field of $A$ then $\Res_{A/k}Z$ is $\mathbb{A}^1$-chain connected as well.

\begin{lemma}\label{lem:sym-sigma-is-flat}
Let $k$ be any field, and $\sigma \colon \Gal(k^s/k) \to \Sigma_n$ any continuous group homomorphism. Then the $n$th symmetric power operation over the absolute closure
\[
\Sym \colon (\mathbb{P}^1_{k^s})^n \to \mathbb{P}^n_{k^s}
\]
descends to give a well-defined, flat morphism of $k$-varieties (that is, an $\mathbb{A}^1$-chamber)
\begin{align*}
    \Sym_\sigma \colon (\mathbb{P}^1_k)^n_\sigma \to \P^n_k.
\end{align*}
\end{lemma}
\begin{proof} By Galois descent, it suffices to see that $\Sym$ is $\Gal(k^s/k)$-equivariant with respect to the twisted action on the left and the ordinary Galois action on the right. Indeed this is immediate by $\Sigma_n$-equivariance of symmetric products, so this descends to a map from the Weil restriction. Since $\Spec(k^s) \to \Spec(k)$ is faithfully flat, and flatness satisfies fpqc descent, the resulting map $\Sym_\sigma$ is flat as well.
\end{proof}
We now show that the maps $\Sym_\sigma$ define chambers with respect to $\Zdomino$.

\begin{lemma}\label{lem:sym-is-domino-chamber} In the notation of \Cref{lem:sym-sigma-is-flat}, $\Sym_\sigma$ is an $\A^1$-chamber of the Wronski map relative to $\Zdomino$.
\end{lemma}
\begin{proof} The map is well-defined and flat by \Cref{lem:sym-sigma-is-flat} and $(\P^1_k)^n_\sigma$ is $\A^1$-connected by \Cref{prop:p1sigma-a1-connected}. The locus $\Zdomino$ is irreducible and is not contained in the ramification locus of the Wronski map by \Cref{prop:Zdomino-not-contained-in-ramification-locus}.

Since $f(\Zdomino) = \Delta_n$ is the discriminant locus, it remains to show that the equation for the discriminant $\Delta_n$ pulls back to a global unit times a square. In fact it will pull back to a scalar times a square. For simplicity, we work in the affine chart $(\A^1)^n_\sigma$. Over $k^s$, we may write the pullback of the usual equation $f_{\Delta_n}$ for the discriminant as 
\[\Sym_\sigma^* f_{\Delta_n} = \prod_{i < j}(t_i - t_j)^2,\]
where the $t_i$ are the roots of the Wronskian. We define
\[\sqrt{\Delta_n} = \prod_{i < j} (t_i - t_j).\]
We are interested to see if this is actually defined on $(\A^1)^n_\sigma$ --- that is, whether it is fixed under the action of the Galois group from \Cref{def:sigma-action}.

Since the coefficients of $\sqrt{\Delta_n}$ are integers, they are fixed by the Galois action. An element $g \in \Gal(k^s/k)$ only acts on $\sqrt{\Delta_n}$ by permuting its factors and introducing a sign
\[g \cdot \sqrt{\Delta_n} = (-1)^{\sigma(g)} \sqrt{\Delta_n},\]
where $(-1)^{\sigma(g)}$ is the sign of the permutation $\sigma(g) \in \Sigma_n$. If $\sigma \colon \Gal(k^s/k) \to \Sigma_n$ lands inside the alternating group, this sign is always positive, and we can conclude that $\sqrt{\Delta_n}$ is $\sigma$-invariant, hence well-defined on $(\P^1)^n_\sigma$, and we conclude that $\Sym_\sigma^\ast f_{\Delta_n} = \sqrt{\Delta_n}^2$.

If $\sigma$ does not land in the alternating group, we obtain a surjection
\[
    \Gal(k^s/k) \xto{\sigma} \Sigma_n \tto \{\pm1\}.
\]
This factors through a quotient $\Gal(F/k)$, where $F/k$ is quadratic. Since $\mathrm{char}(k) \ne 2$, there exists $\alpha \in F$ such that $g\cdot \alpha = (-1)^{\sigma(g)} \alpha$ holds for all $g\in \Gal(k^s/k)$, e.g. by writing $F = k(\sqrt{c})$ for some nonsquare $c \in k$ and putting $\alpha = \sqrt{c}$. Then we see that
\[g \cdot (\alpha \, \sqrt{\Delta_n}) = (-1)^{\sigma(g)} \alpha \cdot (-1)^{\sigma(g)} \sqrt{\Delta_n} = \alpha \sqrt{\Delta_n},\]
so $\alpha \sqrt{\Delta_n}$ is $\sigma$-invariant and we conclude that
\[\Sym_\sigma^* f_{\Delta_n} = \tfrac{1}{c} (\alpha \sqrt{\Delta_n})^2. \qedhere\]
\end{proof}

\begin{example}\label{ex:twisted-discriminant}
Consider the case of $(\P^1_\R)^2_\sigma$ with the nontrivial twisted real structure of \Cref{ex:twisted-real-p1-2}. Let $t_1, t_2$ denote the roots of a quadratic polynomial $(z-t_1)(z-t_2)$, with discriminant $(t_1 - t_2)^2$. By \Cref{ex:twisted-coordinate-ring}, the coordinate ring is generated by $t_1 + t_2$ and $i(t_1 - t_2)$, so we can write the pullback of the discriminant as 
\[(t_1 - t_2)^2 = (-1) \cdot (i(t_1 - t_2))^2,\]
so it is a unit times the square of $i(t_1 - t_2)$.
\end{example}

\subsection{Twisted degrees of the Wronski map} We can now prove our main theorem, recalling the notation we've fixed in \Cref{nota:rho-domino}.
\begin{theorem}\label{thm:main} Let $m,d\ge 0$, let $k$ be a field of characteristic 0 or $\ge m+d$, and let $\sigma \colon \Gal(k^s/k) \to \Sigma_{dm}$ be a continuous homomorphism. Then for any two $k$-points $w_1,w_2 \in \P^{dm}_k$ in the complement of the discriminant, whose roots have Galois type $\sigma$, the associated twisted degrees of the Wronski computed over their fiber are equal:
\begin{align*}
    \deg^{\sdomino}(\Wr,\rhodomino)(w_1) = \deg^{\sdomino}(\Wr,\rhodomino)(w_2) \text{ in } \GW(k).
\end{align*}
That is, for each $\sigma$ there is a well-defined $\sigma$\emph{-twisted degree} of the Wronski map $\deg^\sigma(\Wr) \in \GW(k)$. Over the reals, the sign of $\deg^\sigma(\Wr)$ is the degree of the Wronski map computed relative to the character orientation in \cite{LevinsonPurbhoo}. We additionally have an untwisted degree of the Wronski map computable on any chamber whose sign over the reals is equal to that computed in \cite{eremenkoDegreesRealWronski2002}.
\end{theorem}
\begin{proof} By \Cref{lem:sym-is-domino-chamber}, $\Sym_\sigma$ is an $\A^1$-chamber of the Wronski map relative to $\Zdomino$ in the sense of \Cref{def:relative-A1-chamber}. Thus this result follows from fiber-invariance of twisted degrees over a chamber \Cref{thm:fiber-invariance-twisted-degrees-on-chambers}.

Comparison with \cite{LevinsonPurbhoo} follows immediately from the definition of the character degree in their paper. Comparison with \cite{eremenkoDegreesRealWronski2002} comes from passing to an affine open and carrying out the computation there, using the fact that orientations on affine space are well-defined up to a scalar from the base field.
\end{proof}

\begin{remark} Note that we haven't explicitly chosen $\rhodomino$ of $\sdomino$ in \Cref{nota:rho-domino}. However, since the Grassmannian is proper and geometrically connected and has $\Pic[2]=0$, any two choices of $\rhodomino$ and/or of $\sdomino$ differ by a global scalar, which will affect our degree computations only by multiplication by a global square class (see \Cref{rmk:proper-geometrically-connected-twisted-deg-independent-of-everything-except-D}). In \cite{eremenkoDegreesRealWronski2002}, this is witnessed by the fact that all of their degrees have a sign ambiguity $\pm$. In \cite{LevinsonPurbhoo} this choice is absorbed into the character orientation, as two choices of $\sdomino$ likewise differ by a global scalar. Thus all the degrees we compute in the following section should be understood as being defined up to multiplication by some square class from $k^\times$.
\end{remark}

\begin{remark}
Since our $\A^1$-chambers are all Galois twists of $(\P^1_k)^n$, our construction can be thought of as a morphism $f : X \to Y$ of $k^s$-varieties with many possible $k$-rational structures, all oriented over $k$ with a common orientation $\rho$, but where $\deg(f)$ varies with the choice of twisted rational structure.
\end{remark}

%% file: sections/examples.tex
\section{Examples}\label{sec:examples}
Here we explore a few examples of Wronski maps and compute their degrees (where the parity works out) as well as their twisted degrees corresponding to various choices of Galois twists.

\subsection{The Grassmannian \texorpdfstring{$\Gr(2,4)$}{Gr(2,4)}}

We prove the following formula:

\begin{proposition}\label{prop:Gr24} Let $k$ be a field of characteristic $\ne 2,3$, and let $\Delta_4\subseteq \P^4$ denote the discriminant locus, cut out by the standard formula $f_{\Delta_4}$. Then
\[
    \deg^{\sdomino}(\Wr,\rhodomino) = \langle 3,f_{\Delta_4} \rangle
\]
on $\sheafGW(\P^4\setminus \Delta)$.
\end{proposition}
\begin{proof} Since $\P^4$ is smooth and $\sheafGW$ is unramified, it suffices to work on an open affine where we have deleted the discriminant locus. We can therefore work on the affine Schubert cell $X_\emptyset \subset \Gr(2,4)$ defined by the affine coordinates $(a_{11},a_{12}, a_{21},a_{22})$, where
\[
X_\emptyset = \left\{ \begin{pmatrix}
    1 & 0 & a_{11} & a_{12} \\
    0 & 1 & a_{21} & a_{22}
\end{pmatrix} \right\}.
\]
Rescaling the Wronski polynomial of a point in $X_\emptyset$ to be monic, we obtain that the Wronski polynomial is
\[
\Wr = t^4 + 2a_{21} t^3 + (3a_{22} - a_{11})t^2 - 2a_{12} t + a_{11}a_{22} - a_{12}a_{21}.
\]
The Jacobian of the Wronski map is then $\Jac(\Wr) = 4a_{11} + 12 a_{22}$, and the domino variety is cut out by the equation
\[
\sdomino = a_{11}^2 a_{22} - a_{11} a_{12} a_{21} + a_{11} a_{21}^2 a_{22} - 2a_{11} a_{22}^2 + a_{12}^2 - a_{12} a_{21}^3 + 3a_{12} a_{21} a_{22} + a_{22}^3.
\]
Over a general point $t^4 + at^3 + bt^2 + ct + d$ in $\A^4 \subseteq \P^4$, we can see by direct computation that the local twisted degrees at the fibers are given by
\begin{align*}
    \langle 2/27b^4 - 4/9ab^2c + 2/3a^2c^2+ 16/9b^2d - 16/3acd + 32/3d^2 \\
    \pm (- 2/27b^3 + 1/3abc - a^2d  - c^2 + 8/3bd )\sqrt{b^2 - 3ac + 12d}\rangle.
\end{align*}
If $b^2-3ac+12d$ is not a square in $k$, we can trace down one of the degrees above from its field of definition to compute the global degree \cite{brazeltonTraceLocalBbb2021}. In particular, even if $b^2-3ac+12d$ is a square, we will obtain the same answer regardless by the relations in $\GW(k)$.\footnote{In general along a quadratic field extension, $\Tr_{k(\sqrt{c})/k}\langle a+b\sqrt{c}\rangle = \langle 2a, 2ac(a^2-b^2c)\rangle$. If $c$ has square roots $\pm r$ in $k$, then the analogous class $\langle a+br, a-br\rangle$ is again equal to $\langle 2a, 2ac(a^2-b^2c)\rangle$ using the basic relations for generators on $\GW(k)$.} Thus we can assume without loss of generality it is a nonsquare, and take the field trace to obtain
\begin{align*}
    &\langle 4/27(b^2 - 3ac + 12d)^2\rangle +  \langle 4/27^2f_{\Delta_4}\cdot (b^2 - 3ac + 12d)^4\rangle \\
    &= \langle 3,f_{\Delta_4}\rangle. \qedhere
\end{align*}
\end{proof}

\subsection{The Grassmannian \texorpdfstring{$\Gr(2,4)$}{Gr(2,4)} over finite fields}

When $d=m=2$ and working over a finite field $\mathbb{F}_q$, we have five different chambers to consider, coming from the five possible Galois twists $\sigma \colon \Gal(\overline{\mathbb{F}_q}/\mathbb{F}_q) \to \Sigma_4$. As the absolute Galois group of a finite field is pro-cyclic, the data of a Galois twist is the same data as a partition of the number 4. While we know, in each chamber, what answer we expect to get by \Cref{prop:Gr24}, we still find it instructive to work through specific examples.

\begin{notation} Denote by $\deg^\sigma \Wr \in \GW(k)$ the $\sigma$-twisted degree of the Wronski map over $\Gr(d,d+m)$.
\end{notation}

We then have the following:

\begin{example}\label{exa:Gr24-over-Fq} Let $\mathbb{F}_q$ be a finite field of characteristic $\ne 2,3$. Let $d=m=2$, and take $u \in \mathbb{F}_q$ be any nonsquare element. Then the twisted degrees of the Wronski map are as follows:
\begin{center}
    \begin{tabular}{|r | c c c c c|}
        \hline
        $\sigma$ & $1^4$ & $2^1 1^2$ & $2^2$ & $3^1 1^1$ & $4^1$ \\
        \hline
        $\deg^\sigma(\Wr)$ & $\langle 1,3\rangle$ & $\langle 1,3u\rangle$ & $\langle 1,3\rangle$ & $\langle 1,3\rangle$ & $\langle 1,3u \rangle$\\
        \hline
    \end{tabular}
\end{center}
\end{example}
Via \Cref{thm:main}, twisted degrees depend only upon the field of definition of the roots of the Wronski polynomial, so for each $\sigma$ it suffices to find a convenient degree four polynomial over which to carry out the computation. The subtlety comes from finding a polynomial which works for all finite fields.

\fbox{$\sigma = 1^4$:} Pick any point $c\ne -1,0,1$ in $\mathbb{F}_q$, and consider the Wronski polynomial $(t+1)t(t-1)(t-c)$. The fibers of the Wronski map over this quadratic polynomial can be computed as
\begin{align*}
    L &= \begin{pmatrix} 1 & 0 & \frac{1- \sqrt{1 + 3c^2}}{2} & - \frac{c}{2} \\
    0 & 1 & - \frac{c}{2} & \frac{-1 + \sqrt{1+3c^2}}{6}\end{pmatrix}\\
    L' &= \begin{pmatrix} 1 & 0 & \frac{1- \sqrt{1 - 3c^2}}{2} & - \frac{c}{2} \\
    0 & 1 & - \frac{c}{2} & \frac{-1 - \sqrt{1+3c^2}}{6}\end{pmatrix}.
\end{align*}
Whether these solutions are $k$-rational or not depends upon whether $1+3c^2$ is a square class in our base field $k$, so without loss of generality let's assume that $1+3c^2$ is a non-square. In that case, we transfer the twisted local degree down from the quadratic field extension, and compute that the global twisted degree is
\begin{align*}
    \Tr_{\mathbb{F}_q(\sqrt{3c^2+1})/\mathbb{F}_q} \left\langle\frac{2}{27} \left(9 c^4+\left(6+9 \sqrt{3 c^2+1}\right) c^2-\sqrt{3 c^2+1}+1\right) \right\rangle = \langle 1,3\rangle.
\end{align*}

\fbox{$\sigma = 2^11^2$:} Let $u\in \mathbb{F}_q$ be a nonsquare class, and consider the Wronski polynomial $(t^2-u)(t^2-1)$ with roots at $\pm1$ and $\pm\sqrt{u}$. The field of definition of the fibers is $\mathbb{F}_q(\sqrt{u^2 + 14u + 1})/ \mathbb{F}_q$, and we can compute that the twisted degree is equal to
\begin{align*}
    &\Tr_{\mathbb{F}_q(\sqrt{u^2 + 14u + 1})/ \mathbb{F}_q} \left\langle 2u^4 + 56u^3 + 396 u^2 + 56u + 2 + (2u^3 - 66u^2 - 66u + 2)\sqrt{u^2 + 14u + 1} \right\rangle \\
    &= \langle 1,3u\rangle.
\end{align*}
\fbox{$\sigma = 2^2$:} Again letting $u$ be a nonsquare, we can take the Wronski polynomial with roots at $\pm \sqrt{u}$ and $1\pm \sqrt{u}$:
\begin{align*}
    \Wr(t) = (t^2-u)(t^2-2t+1-u).
\end{align*}
Again we see the local twisted degrees are Galois conjugate over the field of definition of the fibers $\mathbb{F}_q(\sqrt{16u^2-4u+1})$, so we can trace down one of the degrees to get that $n_{2,2,\sigma}$ is
\begin{align*}
    &\Tr_{\mathbb{F}_q(\sqrt{16u^2-4u+1})/\mathbb{F}_q} \left\langle 512u^4 - 256u^3+ 96u^2- 16u+ 2 + (- 128u^3  + 48u^2  + 12u  - 2) \sqrt{16u^2{-}4u{+}1}\right\rangle \\
    &= \langle 1,3\rangle.
\end{align*}

\begin{remark} The previous three cases are observable over the reals, and the signatures of the resulting forms can be seen in \cite{LevinsonPurbhoo}. The following two cases do not appear over the real numbers, but are nevertheless computable from our general framework.
\end{remark}

\fbox{$\sigma = 3^11^1$:} Let $t^3 + at + b$ be an irreducible polynomial over $\mathbb{F}_q$. Then we can work with the Wronski polynomial
\begin{align*}
    W(t) = t(t^3 + at + b).
\end{align*}
The fibers of this turn out to be rational, and the sum of their twisted degrees is
\begin{align*}
    \langle -ab^2, \frac{4}{27}a^4 + ab^2 \rangle.
\end{align*}
Since the cubic factor in $W(t)$ is irreducible, we know that its discriminant $-4a^3 - 27b^2$ is a nonzero square in $\mathbb{F}_q$. Observe then that
\begin{align*}
    \frac{4}{27}a^4 + ab^2 = \frac{-a}{27}(-4a^3 - 27b^2) \equiv -3a.
\end{align*}
Thus the form above reduces to
\begin{align*}
    \langle -a, -3a\rangle = \langle 1,3 \rangle.
\end{align*}
\fbox{$\sigma = 4^1$:} We have to try to find an irreducible degree four polynomial over a large family of finite fields. Observe that if a polynomial $t^4 + at + b$ is irreducible over $\mathbb{F}_q$ then its discriminant $256b^3 - 27a^4$ is a nonsquare in $\mathbb{F}_q$. We can work over such a polynomial
\[
    W(t) = t^4 + at + b,
\]
and we get that the local twisted degrees are defined over $\mathbb{F}_q(\sqrt{3b})$, and are of the form
\begin{align*}
    \langle 32b^2 \pm 6a^2 \sqrt{3b}\rangle.
\end{align*}
Assuming $3b$ is a nonsquare, the global twisted degree is
\begin{align*}
    \Tr_{\mathbb{F}_q(\sqrt{3b})/\mathbb{F}_q}\langle 32b^2 + 6a^2 \sqrt{3b}\rangle &= \langle 1, -768(27a^4 - 256b^3)b^4
 \rangle \\
 &= \langle 1, 3(-27a^4 + 256b^3)\rangle.
\end{align*}
\begin{remark} Another example to consider is the cyclotomic polynomial
\begin{align*}
    \Phi_4(t) = t^4 + t^3 + t^2 + t + 1.
\end{align*}
This is irreducible over $\mathbb{F}_q$ precisely when $q\equiv 2,3\pmod{5}$ (which in particular guarantee $5$ is a nonsquare in $\mathbb{F}_q$). The fibers of the Wronski map involve a square root of 10, so we have to break into cases depending on whether 10 is a square. If 10 is a nonsquare, then the global twisted degree can be computed as
\begin{align*}
    \Tr_{\mathbb{F}_q(\sqrt{10})/\mathbb{F}_q} \langle 200 + 25\sqrt{10}\rangle &= \langle 1,15\rangle = \langle 3,5\rangle.
\end{align*}
This will agree with the computation above exactly when $15$ and $3u$ have the same square class, that is, when $5$ is a nonsquare, which we know.

On the other hand if $10$ happened to be a square, the global twisted degree over $\Phi_4(t)$ would become
\begin{align*}
    \langle 200 + 25\sqrt{10}\rangle + \langle 200 - 25\sqrt{10}\rangle &= \langle 8+\sqrt{10}\rangle + \langle 8 -\sqrt{10}\rangle = \langle 1,54\rangle = \langle 1,6\rangle.
\end{align*}
This agrees with $\langle 1,3u\rangle$ when 2 is a nonsquare, which is guaranteed by the hypotheses that 5 is a nonsquare and 10 is a square.
\end{remark}

\begin{remark}
The degree of the complex Wronski map for $m=2,d=3$ is $5$, and we know by \cite{purbhooJeuTaquinMonodromy2010} that the Galois groups of osculating Schubert problems are full symmetric whenever $d\ne m$.\footnote{In \cite[5.6]{purbhooJeuTaquinMonodromy2010}, Purbhoo exhibits certain classes $s_{k,L}$ in the monodromy group of the real Wronski map. These are bijections on the set $\SYT(d\times m)$, and fail to generate the full symmetric group on this set only for square tableaux, i.e. when $d=m$.} Unfortunately this means that trying to replicate \Cref{prop:Gr24} will be much harder, as there is no formula in radicals for the general fibers over a Wronski polynomial of degree 6. Nevertheless we know by \cite{LevinsonPurbhoo} that the answer cannot depend only upon the square class of the discriminant (for instance the loci $1^5$ and $2^21^2$ have different twisted degrees over the reals but the discriminant is positive in both chambers) hence it is necessarily much more complicated than \Cref{prop:Gr24}.
\end{remark}